\documentclass[a4paper,11pt]{amsart}

\usepackage[active]{srcltx}
\usepackage{ifpdf}
\ifpdf 
    \usepackage[pdftex,     
            plainpages=false,   
            breaklinks=true,    
            colorlinks=true,
            pdftitle=My Document
            pdfauthor=My Good Self
           ]{hyperref} 
\else 
    \usepackage{hyperref}       
\fi

\usepackage[utf8]{inputenc}
\usepackage[T1]{fontenc}
\usepackage{amsmath,amssymb,amsfonts,mathtools,amsthm}
\usepackage{color}
\usepackage{esint}
\usepackage{enumerate}
\usepackage{cleveref}

\theoremstyle{plain}
\newtheorem{thm}{Theorem}[section]
\newtheorem{lem}[thm]{Lemma}
\newtheorem{prop}[thm]{Proposition}

\theoremstyle{definition}
\newtheorem*{rem}{Remark}
\newtheorem*{defi}{Definition}

\def\R{\mathbb{R}}
\def\N{\mathbb{N}}
\def\Z{\mathbb{Z}}

\def\Td{\mathbb{T}^d}
\def\div{\operatorname{div}}
\def\D{\mathcal{D}}
\def\adiv{\mathcal{R}}

\def\e{\varepsilon}

\numberwithin{equation}{section} 

\title[Convex integration for the transport equation]{Convex integration solutions to the transport equation with full dimensional concentration}
\date{\today}
\author{Stefano Modena}
\address{Mathematisches Institut, Universit\"at Leipzig, D-04109 Leipzig, Germany}
\email{stefano.modena@math.uni-leipzig.de}
\author{Gabriel Sattig}
\address{Mathematisches Institut, Universit\"at Leipzig, D-04109 Leipzig, Germany}
\email{gabriel.sattig@math.uni-leipzig.de}

\begin{document}

\begin{abstract}
We construct infinitely many incompressible Sobolev vector fields $u \in C_t W^{1,\tilde p}_x$ on the periodic domain $\mathbb{T}^d$ for which uniqueness of solutions to the transport equation fails in the class of densities $\rho \in C_t L^p_x$, provided $1/p + 1/\tilde p > 1 + 1/d$. The same result applies to the transport-diffusion equation, if, in addition, $p’<d$. 
\end{abstract}

\maketitle

\section{Introduction}

This paper deals with the problem of (non)uniqueness of solution to the Cauchy problem for the transport equation
\begin{equation}
\label{eq:transport-eqn}
\partial_t \rho + \nabla \rho \cdot u  =0, 
\end{equation}
on the $d$-dimensional flat torus $\Td := \R^d / \Z^d$, where $u: [0,T] \times \Td \to \R^d$ is a given (locally integrable) vector field  and $\rho: [0,T] \times \Td \to \R$ is the unknown density. We will always assume that $u$ is incompressible, i.e.
\begin{equation}
\label{eq:incompressibility}
\div u = 0,
\end{equation}
in the sense of distributions. Under this condition, \eqref{eq:transport-eqn} is formally equivalent to the continuity equation
\begin{equation}
\label{eq:continuity-eqn}
\partial_t \rho + \div_x (\rho u) = 0.
\end{equation}
We prove the following theorem. 

\begin{thm}
\label{thm:main-weak-form}
Let $ p\in[1,\infty), \tilde p \in [1, \infty) $, and assume that
\begin{equation}
\label{eq:exponents-main-condition}
\frac{1}{p} + \frac{1}{\tilde p} > 1 + \frac{1}{d}.
\end{equation}
Then there are infinitely many incompressible vector fields satisfying
\begin{equation}
\label{eq:thm-u}
u\in C_t L^{p'}_x \cap C_t W^{1, \tilde p}_x
\end{equation}
for which uniqueness of distributional solutions to the transport equation \eqref{eq:transport-eqn} fails in the class of densities
\[ \rho\in C_t L^p_x.\]
Moreover, if $p = 1$, it holds $u \in C ( [0,T] \times \Td )$.
\end{thm}
Here and in the following we will use the notation $C_t L^p_x := C([0,T], L^p(\Td))$, and, similarly, $L^r_t L^p_x := L^r((0,T), L^p(\Td))$. 
\begin{rem}
As a matter of fact, one can strengthen condition \eqref{eq:thm-u} and produce vector fields which satisfy
\begin{equation*}
u \in C_t L^{p'}_x \cap \bigcap_{\substack{\tilde p \text{ such that} \\ \eqref{eq:exponents-main-condition} \text{ holds}}} C_t W^{1, \tilde p}_x
\end{equation*}
and, moreover, $\|u\|_{L^{p'}} \leq \e$, for any fixed $\e>0$. See Theorem \ref{thm:weak} below.  We mention also that Theorem \ref{thm:main-weak-form} can be extended to cover the case of the transport-diffusion equation and to produce more regular densities and fields, provided more restrictive conditions on the exponents $p, \tilde p$ are assumed. See Theorems \ref{thm:diffusion} and \ref{thm:higherReg} below for the precise statements. 
\end{rem}

\subsection{Background}

It is well known that, when $u$ is at least Lipschitz continuous (in the space variable), the solution to \eqref{eq:transport-eqn} is given by the implicit formula
\begin{equation}
\label{eq:solution-smooth-case}
\rho \left( t, X(t,x) \right) = \rho(0,x),
\end{equation}
where $X(t,x)$ is the flow solving the ODE
\begin{equation}
\label{eq:ode}
\begin{aligned}
\partial_t X(t,x) & = u \left( t, X(t,x) \right), \\
X(0,x) & = x.
\end{aligned}
\end{equation}

It is in general of great importance, both for theoretical interest and for the applications to many physical models, to study the well posedness of the Cauchy problem \eqref{eq:transport-eqn}, in the case the vector field $u$ is not smooth, i.e. less then Lipschitz continuous. 

There are several ways to state the well posedness problem in the weak setting. The one we propose here is one possibility. We refer to \cite{Modena2018} for a more comprehensive discussion. Fix an exponent $p \in [1, \infty]$ and denote by $p'$ its dual H\"older
\begin{equation*}
\frac{1}{p} + \frac{1}{p'} = 1.
\end{equation*}
We ask two questions.

\begin{enumerate}[(a)]
\item Do existence and uniqueness of distributional solutions to \eqref{eq:transport-eqn} hold in the class of densities 
\begin{equation}
\label{eq:class-density}
\rho \in L^\infty_t L^p_x
\end{equation}
for a given vector field
\begin{equation}
\label{eq:class-field}
u \in L^1_t L^{p'}_x \, ?
\end{equation}

\item Is the relation \eqref{eq:solution-smooth-case} still valid, in some weak sense? In other words, is there still a connection between the Lagrangian world \eqref{eq:ode} and the Eulerian one \eqref{eq:transport-eqn}? 
\end{enumerate}

Let us observe that the choice of the class \eqref{eq:class-density} is motivated by the fact that, for smooth solutions of \eqref{eq:transport-eqn}-\eqref{eq:incompressibility}, every $L^p$ norm is constant in time: it is thus reasonable to expect that,  for weak solutions, the $L^p$ norm, if not constant, remains, at least, uniformly bounded in time. Once the class for the density \eqref{eq:class-density} is fixed, the choice \eqref{eq:class-field} for the vector field is natural, because in this way the product $\rho u \in L^1 ( (0,T) \times \Td)$ and thus the transport equation \eqref{eq:transport-eqn}, in its equivalent form \eqref{eq:continuity-eqn}, can be considered in distributional sense.

We list now some answers to the questions (a), (b) above, which can be found in the literature. The first consideration is that the \emph{existence} of distributional solutions is a pretty easy task. Indeed, regularizing the vector field and the initial datum, one can use the classical theory for ODE and formula \eqref{eq:solution-smooth-case} to produce a sequence of approximate solutions, which turns out to be uniformly bounded in $L^\infty_t L^p_x$. From such sequence one can then extract a weakly converging subsequence, whose limit is a solution to \eqref{eq:transport-eqn}, because of the linearity of the equation.


Let us now discuss some \emph{uniqueness} results. In their groundbreaking paper  \cite{DiPerna:1989vo}, R.~DiPerna and P.L~Lions proved that, for every $p \in [1, \infty]$, uniqueness holds in the class of densities \eqref{eq:class-density} for a given vector field $u$ as in \eqref{eq:class-field}, provided, in addition,
\begin{equation}
\label{eq:diperna-lions}
u \in L^1_t  W^{1, p'}_x. 
\end{equation}
Moreover, the incompressibility assumption can be substituted by the weaker requirement $\div u \in L^\infty$ (see also \cite{SEIS20171837} for a further relaxation in the case of the continuity equation).  DiPerna and Lions' proof is based on a regularization argument. Denoting by $\rho^\e$,  $u^\e$ a standard mollification of $\rho$ and $u$, the equation for $\rho^\e$, $u^\e$ reads,
\begin{equation*}
\partial_t \rho^\e + \div(\rho^\e u^\e) = r^\e,
\end{equation*}
where $r^\e$ is the \emph{commutator} $r^\e = \div \left(\rho^\e u^\e - (\rho u)^\e \right)$, given by the fact that the mollification of the product is not equal, in general, to the product of the mollifications. After some manipulation, it can be shown that $r^\e$ has the form
\begin{equation*}
r^\e \approx \rho^\e \, \nabla u^\e,
\end{equation*}
i.e. it is the product of the density and the derivative of the vector field. Such expression suggests, in some sense, that the commutator converges to zero as $\e \to 0$ (and thus uniqueness of solutions holds), for a density $\rho \in L^\infty_t L^p_x$, provided $\nabla u \in L^1_t L^{p'}_x$, which is exactly DiPerna and Lions' condition \eqref{eq:diperna-lions}. In other words, the interplay between the integrability of the density and the integrability of the derivative of the vector field plays a crucial role: very roughly speaking, a Sobolev vector field is ``Lipschitz like'' on a very large set, and there is just a very small ``bad'' set, where $\nabla u$ can be very large. A density $\rho$ with integrability $L^p$ that ``matches'' the integrability $L^{p'}$ of $\nabla u$ \emph{does not see} the bad set of $u$, and this implies uniqueness.
 
A natural question is now whether it is possible to lower the regularity \eqref{eq:diperna-lions} of $u$ and still have uniqueness of solutions in $L^\infty_t L^p_x$. 

In the class of \emph{bounded} densities, (i.e. $p=\infty$ in our notation), L.~Ambrosio \cite{Ambrosio:2004cva} showed in 2004 that uniqueness holds if the vector field $u \in L^1((0,T), BV(\Td))$ and it has bounded divergence, whereas S. Bianchini and P. Bonicatto in \cite{Bianchini:2017vf} were able to prove uniqueness in the $BV$ framework for the more general class of \emph{nearly incompressible} vector fields. 

Concerning question (b) above, it is a general principle in the theory of the transport equation that, whenever existence and uniqueness for the PDE \eqref{eq:transport-eqn} holds in the class of \emph{bounded} densities, then existence and uniqueness holds also for the ODE \eqref{eq:ode}, in the sense of the \emph{regular Lagrangian flow} and, moreover, the bridge \eqref{eq:solution-smooth-case} between the Lagrangian world and the Eulerian one still holds true. We refer to \cite{Ambrosio2017} for a detailed discussion in this direction. 

From the analysis above, it follows that the uniqueness results present in the literature are based essentially on two assumptions on the vector field: on one side, a bound on the derivative $Du$ is needed (e.g. $u$ Sobolev or $BV$); on the other side, a condition on the divergence of $u$ is required (e.g. $\div u =0$, or $\div u \in L^\infty$, or $u$ nearly incompressible). 

The most part of the counterexamples to uniqueness that can be found in the literature are based on the absence of at least one of those two conditions. There are counterexamples to uniqueness with Sobolev vector field with unbounded divergence (e.g. in DiPerna and Lions' paper \cite{DiPerna:1989vo}), and there are counterexamples to uniqueness for incompressible vector fields, which do not possess one full derivative (e.g. $u \in W^{s,1}$ for every $s<1$, but $u \notin W^{1,1}$), see, for instance, \cite{DiPerna:1989vo}, \cite{Depauw:2003wl}.  All such counterexamples are based on the failure of uniqueness at a \emph{Lagrangian} level: one constructs a pathological vector field for which the ODE admits two different flows of solutions and then uses such flows to produce non-unique solutions to the PDE: once again, the connection \eqref{eq:solution-smooth-case} is crucial. 

\subsection{Non-uniqueness for Sobolev vector fields and our contribution}

The mentioned counterexamples, therefore, do not answer the question whether uniqueness holds in the class of densities \eqref{eq:class-density}, if
\begin{equation}
\label{eq:new-framework}
\text{ $u$ is incompressible, $u \in L^1_t W^{1, \tilde p}_x$, but $\tilde p < p'$.}
\end{equation}
In such framework there are two competing mechanisms. On one side, by DiPerna and Lions result, uniqueness holds, at least, in the class of bounded densities, and thus, by the observation made before, uniqueness at the Lagrangian level is satisfied (again in the sense of the regular Lagrangian flow): in other words, the vector field is very well behaved from the ODE point of view. On the other side, the integrability of $\rho$ and the of $Du$ do not ``match'' anymore and thus, referring the the heuristic introduced above, it could happen that an $L^p$ density ``sees the bad set'' of a $W^{1, \tilde p}$ vector field, so that purely Eulerian non-uniqueness phenomena could appear. 

The framework \eqref{eq:new-framework} was considered, for the first time, quite recently in \cite{Modena2018} and \cite{modena-szekelyhidi18}, where the analog of Theorem \ref{thm:main-weak-form} was proven, with assumption \eqref{eq:exponents-main-condition} substituted by the strongest assumption 
\begin{equation}
\label{eq:exponents-ms}
\frac{1}{p} + \frac{1}{\tilde p} > 1 + \frac{1}{d-1},
\end{equation}
using a convex integration approach and exploiting a \emph{concentration} mechanism, in the spirit of the \emph{intermittency} added to the convex integration schemes by T.~Buckmaster and V.~Vicol in \cite{Buckmaster:2017wf}. 

Our main result, namely Theorem \ref{thm:main-weak-form}, shows that such approach can be extended to produce examples of non-uniqueness for the transport equation with \emph{full dimensional concentration}, i.e. with $d$ instead of $d-1$ in \eqref{eq:exponents-ms}. Notice that the result in \cite{Modena2018, modena-szekelyhidi18} and our Theorem \ref{thm:main-weak-form}  in particular implies that the duality between Lagrangian and Eulerian world is completely destroyed, even for Sobolev and incompressible (thus, quite ``well behaved'' vector field): there are many distributional solutions, but only one among them is transported by the regular Lagrangian flow as in \eqref{eq:solution-smooth-case}.

It is still an open question whether uniqueness of weak solutions to \eqref{eq:transport-eqn} holds if the Sobolev integrability $\tilde p$ of the field, $Du \in L^1_t L^{\tilde p}_x$, lies in the range
\begin{equation}
\label{eq:intermediate-range}
1 < \frac{1}{p} + \frac{1}{\tilde p} \leq 1 + \frac{1}{d},
\end{equation}
and thus whether Theorem \ref{thm:main-weak-form} is or is not optimal. Let us nevertheless observe that, for $p=1$, Theorem \ref{thm:main-weak-form} provides existence of continuous vector fields
\begin{equation}
\label{eq:crippa-caravenna}
u \in C_t W^{1, \tilde p}_x
\end{equation}
for every $\tilde p < d$, for which uniqueness fails (in the class $\rho \in C_t L^1_x$). On the other side, in a recent result by L.~Caravenna and G.~Crippa \cite{Caravenna:2016kg, caravenna-crippa18} uniqueness (for $\rho \in L^1_{tx}$) is proven, provided \eqref{eq:crippa-caravenna} is satisfied for some $\tilde p > d$ (in particular $u$ is continuous) and $u$ satisfies the additional assumption of ``uniqueness of forward-backward characteristics''. We refer to \cite{Caravenna:2016kg, caravenna-crippa18} for the precise definition. Such result could suggest that, at least in the case $p=1$, Theorem \ref{thm:main-weak-form} (and in particular condition \eqref{eq:exponents-main-condition}) could be sharp. 

A last point  is worth mentioning. Contrary to other recent results  in convex integration (e.g. \cite{Buckmaster:2017wf, cheskidov19, titi18, luo18}) where \emph{concentration} or \emph{intermittency} have been used, in this paper we use a completely physical space based approach and we deliberately avoid any use of Fourier methods and Littlewood-Paley theory. This has, in our opinion, at least two advantages. First, the paper is completely self contained, in particular we do not use any abstract theorem on Fourier multipliers. Secondly, we think that a proof developed in the physical  space can provide a better understanding of the structure of the ``anomalous'' vector fields we are exhibiting and therefore could help in getting an insight on the relation, if any, between the (very well behaved) Lagrangian structure of the vector fields  and the non-Lagrangian solutions we construct.

\bigskip

We conclude this section observing that the proof of Theorem \ref{thm:main-weak-form} is an immediate consequence of the following more general theorem, whose proof is the main topic of the paper.

\begin{thm}[Solutions for Sobolev vector fields]
\label{thm:weak}
	Let $\varepsilon>0$, let $\bar{\rho}\in C^\infty([0,T]\times\Td)$ with zero mean value in the space variable and let $\bar{u}\in C^\infty([0,T]\times\Td,\R^d)$ be a divergence-free vector field. Set $E\coloneqq\left\lbrace t\in[0,T]:\partial_t \bar{\rho}+\operatorname{div}(\bar{\rho}\bar{u})=0 \right\rbrace$. Let $p\in[1,\infty)$ and define $q\in[1,\infty)$ such that
	\begin{equation}
	\frac{1}{p}+\frac{1}{q}=1+\frac{1}{d}.
	\label{eq:p-q-cond}
	\end{equation}
	Then there are functions $\rho:[0,T]\times\Td\to\R$ and $u:[0,T]\times\Td\to\R^d$ such that
	\begin{enumerate}[(i)]
	\item $\rho\in C\left([0,T],L^p(\Td)\right)$ and $u\in C\left([0,T],L^{p'}(\Td)\right) \cap \bigcap_{\tilde{p}<q}C\left([0,T],W^{1,\tilde{p}}(\Td)\right)$.
	\subitem If $ p=1 $ then $ u $ is also continuous: $u\in C\left([0,T]\times\Td\right) $;
	\item $(\rho,u)$ is a distributional solution of \eqref{eq:continuity-eqn}--\eqref{eq:incompressibility};
	\item $(\rho,u)(t) = (\bar{\rho},\bar{u})(t)$ for all $t\in E$;
	\item $\|\rho(t)-\bar{\rho}(t)\|_{L^p} < \varepsilon$ for all $t\in[0,T]$.
	\end{enumerate}
	Statement (iv) can be replaced by the similar
	\begin{enumerate}[(i')]
		\setcounter{enumi}{3}
		\item $\|u(t)-\bar{u}(t)\|_{L^{p'}} < \varepsilon$ for all $t\in[0,T]$.
	\end{enumerate}
\end{thm}
\noindent From this theorem, Theorem \ref{thm:main-weak-form}, i.e. the non-uniqueness of the transport equation, can be easily deduced.

\begin{proof}[Proof of Theorem \ref{thm:main-weak-form}, assuming Theorem \ref{thm:weak}]
Let $\bar \rho \in C^\infty(\Td)$ with zero mean value but not identically zero. Choose $ \chi:[0,T]\to[0,1] $ smooth such that $ \chi$ is equal to zero on $[0,T/3] $ and one on $[2T/3,T] $. Then the function $ (t,x)\mapsto\chi(t)\bar{\rho}(x) $ is smooth and has zero mean value in $ x $ at any time. We can apply Theorem~\ref{thm:weak} on $\chi\bar{\rho}$ and $\bar{u}\equiv0$ and obtain a solution of the transport equation $ (\rho,u) $ with the claimed regularity. As at times $ t\in [0,\tfrac{T}{3}]\cup[\tfrac{2T}{3},T] $ the transport equation is solved by $ (\chi\bar{\rho},\bar{u}) $ in the strong sense, in particular the initial and final values of $ \rho $ are maintained because of statement (iii) of the theorem. Therefore $\rho|_{t = 0} \equiv 0$ and $\rho|_{t = T} = \bar \rho \not \equiv 0$.
\end{proof}

\subsection{Some comments on the method used in the proof}

The proof of Theorem \ref{thm:weak} is based on a convex integration technique: smooth approximate solutions to the continuity equations are constructed, which in the weak limit produce an exact but only distributional solution. In each iterations step the error is decreased by adding a small oscillating perturbation to both density and velocity field.

In the past years convex integration has been applied very successfully on the Euler equations in order positively prove Onsager's conjecture (see, for instance \cite{Isett:2016to, Buckmaster:2017uz}). However, for obtaining Sobolev vector fields, i.e. fields with one full derivative (in some $L^{\tilde p}$ space) new ideas are required. Inspired by the \emph{intermittent Beltrami flow} used in the \cite{Buckmaster:2017wf} (see also \cite{buckmaster-colombo-vicol18} for the related notion of \emph{intermittent jets}), L.~Sz\'ekelyhidi and the first author  adopted, as  building block of their construction in the mentioned papers \cite{Modena2018, modena-szekelyhidi18}, some stationary solutions to the continuity equation called \emph{concentrated Mikado densities and field}, proving the analog of Theorem \ref{thm:main-weak-form} under the less restrictive assumption \eqref{eq:exponents-ms}. The idea of using ``Mikado flows'' for the equation of fluid dynamics was introduced for the first time by S.~Daneri and L.~Sz\'ekelyhidi  in \cite{SzekelyhidiJr:2016tp}. The ``concentrated'' Mikado are suitable modifications of the standard Mikado, having different scaling in different $L^p$ norms.  The $d-1$ in \eqref{eq:exponents-ms} comes from the fact that Mikado functions depends only on $d-1$ coordinates and thus only a $(d-1)$-dimensional concentration is possible.

In the present paper, we are able to substitute $d-1$ with $d$, as we use, as building block of our construction, suitable approximate solutions to the continuity equation, called \emph{space-time Mikado densities and fields}, see Section \ref{subsec:mikado} for the precise definition. Adding the time dependence to the building block allows, roughly speaking, to gain one further dimension and thus to pass from \eqref{eq:exponents-ms} to \eqref{eq:exponents-main-condition}.

\subsection{Extension to transport-diffusion and to higher regularity}

Similarly to \cite{Modena2018, modena-szekelyhidi18}, Theorem \ref{thm:weak} (and thus also Theorem \ref{thm:main-weak-form}) can be extended to cover the case of the transport-diffusion equation
\begin{equation}
\begin{aligned}
\partial_t\rho + \div(\rho u) -\Delta\rho &= 0, \\
\div u &=0,
\label{eq:transport-diffusion}
\end{aligned}
\end{equation}
provided more restrictive conditions on the exponent $p, \tilde p$ are assumed. 
Roughly speaking, the non-uniqueness produced by the transport term $\div (\rho u)$ (i.e. by the interplay between density and field) can be so strong that it can beat the regularizing effect induced by a diffusion operator (see to \cite{Modena2018} for a more comprehensive discussion on this subject).

\begin{thm}[Analog of Theorem~\ref{thm:weak} for the Transport-diffusion equation]
	\label{thm:diffusion}
	Let $\varepsilon>0$, let $\bar{\rho}\in C^\infty([0,T]\times\Td)$ with zero mean value and let $\bar{u}\in C^\infty([0,T]\times\Td,\R^d)$ be a divergence-free field. Set $E\coloneqq\left\lbrace t\in[0,T]:\partial_t \bar{\rho}+\div(\bar{\rho}\bar{u})-\Delta\bar{\rho}=0 \right\rbrace$. Let $p\in(1,\infty)$ and $\tilde{p}\in [1,\infty)$ such that
	\begin{equation}
	\frac{1}{p} + \frac{1}{\tilde p} > 1 + \frac{1}{d}, \qquad p'<d.
	\label{eq:p-cond-diffusion}
	\end{equation}
	Then there are functions $\rho:[0,T]\times\Td\to\R$ and $u:[0,T]\times\Td\to\R^d$ such that
	\begin{enumerate}[(i)]
		\item $\rho\in C\left([0,T],L^p(\Td)\right)$ and $u\in C\left([0,T],L^{p'}(\Td)\right) \cap C\left([0,T],W^{1,\tilde{p}}(\Td)\right)$;
		\item $(\rho,u)$ is a distributional solution of \eqref{eq:transport-diffusion};
		\item $(\rho,u)(t) = (\bar{\rho},\bar{u})(t)$ for all $t\in E$;
	\item $\|\rho(t)-\bar{\rho}(t)\|_{L^p} < \varepsilon$ for all $t\in[0,T]$.
	\end{enumerate}
	Statement (iv) can be replaced by the similar
	\begin{enumerate}[(i')]
		\setcounter{enumi}{3}
		\item $\|u(t)-\bar{u}(t)\|_{L^{p'}} < \varepsilon$ for all $t\in[0,T]$.
	\end{enumerate}
\end{thm}
\begin{rem}
Notice that \eqref{eq:p-cond-diffusion} in particular requires $d>2$, so we cannot show non-uniqueness for the dissipative equation for $ d=2 $ as in the ``inviscid'' transport equation. 
\end{rem}

Theorems \ref{thm:weak} and \ref{thm:diffusion} can be further generalized to cover the generalized transport-diffusion equation
\begin{equation}
	\begin{aligned}
		\partial_t\rho + \div_x(\rho u) +L_k\rho &= 0, \\
		\div_x u &=0,
		\label{eq:high-order-diffusion}
	\end{aligned}
\end{equation}
where $ L_k $ is any constant-coefficient linear differential operator of grade $ k $ (not necessarily elliptic), and to produce more regular densities and vector fields.

\begin{thm}[Analog for solutions with higher regularity and higher order diffusion]
	\label{thm:higherReg}
	Let $ \varepsilon>0 $, let $\bar{\rho}\in C^\infty([0,T]\times\Td)$ with zero mean value and let $\bar{u}\in C^\infty([0,T]\times\Td,\R^d)$ be a divergence-free field. Let $ p,\tilde{p} \in[1,\infty)$ and $ m,\tilde{m}\in\N $ such that
	\begin{equation}
	\frac{1}{p} + \frac{1}{\tilde{p}} > 1+ \frac{m+\tilde{m}}{d} \ \mathrm{and} \ 
	\tilde{p} < \frac{d}{\tilde{m}+k-1}.
	\label{eq:p-cond-higher}
	\end{equation}
	Then there are $s \in [p, \infty]$ and functions $\rho:[0,T]\times\Td\to\R$ and $u:[0,T]\times\Td\to\R^d$ such that
	\begin{enumerate}[(i)]
		\item $\rho \in C([0,T], L^s(\Td))$, $u \in C([0,T], L^{s'}(\Td)$ and, moreover,  $\rho\in C\left([0,T],W^{m,p}(\Td)\right)$, $u\in C\left([0,T],W^{\tilde{m},\tilde{p}}(\Td)\right)$;
		\item $(\rho,u)$ is a distributional solution of \eqref{eq:high-order-diffusion};
		\item $(\rho,u)(t) = (\bar{\rho},\bar{u})(t)$ for all $t\in E$ defined as in Theorem~\ref{thm:weak};
	\item $\|\rho(t)-\bar{\rho}(t)\|_{L^s} < \varepsilon$ for all $t\in[0,T]$.
	\end{enumerate}
	Statement (iv) can be replaced by the similar
	\begin{enumerate}[(i')]
		\setcounter{enumi}{3}
		\item $\|u(t)-\bar{u}(t)\|_{L^{s'}} < \varepsilon$ for all $t\in[0,T]$.
	\end{enumerate}
\end{thm}

\begin{rem}
Observe also that, if we choose $m=0$, $\tilde m = 1$, $k = 2$ in Theorem \ref{thm:higherReg}, the first condition in \eqref{eq:p-cond-higher} reduces to  the first condition in \eqref{eq:p-cond-diffusion}, nevertheless \eqref{eq:p-cond-higher} is not equivalent to \eqref{eq:p-cond-diffusion}. Indeed \eqref{eq:p-cond-diffusion} implies \eqref{eq:p-cond-higher}, but the viceversa is not true, in general. This can be explained by the fact that Theorem \ref{thm:diffusion}, for any given $p$, produces a vector field $u \in C_t L^{p'}_x$, whereas Theorem \ref{thm:higherReg} produces $u \in C_t L^{s'}_x$ for some $s' \leq p'$. 
\end{rem}

\begin{rem}
In Section \ref{sec:prop} we state the main Proposition of this paper, namely Proposition \ref{prop:main}, and we show how Theorem \ref{thm:weak} can be deduced from Proposition \ref{prop:main}. In Sections \ref{sec:tools}-\ref{sec:proof} we give a complete proof of Proposition \ref{prop:main},  assuming $p > 1$, for the sake of simplicity. In Section \ref{sec:appendix} we give a sketch of the proof of Proposition \ref{prop:main} in the case $p=1$ as well as a sketch of the proofs of Theorems \ref{thm:diffusion} and \ref{thm:higherReg}. 
\end{rem}
%
%

\subsection{Notations}
We fix some notations which will be used throughout the paper. 
\begin{itemize}
	\item Integrals, $ L^p $-norms and Sobolev norms of functions defined on $ [0,T]\times\Td $ will always be evaluated on the space $ \Td $ at a single time $ t $, we will write
	\[ \|\rho(t)\|_{L^p} = \|\rho(t,\cdot)\|_{L^p(\Td)} \text{ and } \int_{\Td} \rho = \int_{\Td} \rho(t,x) dx. \]
	\item Similarly, all differential operators (except $ \partial_t $, of course) apply on the space variable: $ \partial_j=\frac{\partial}{\partial_{x_j}},\ \div=\div_x,\ \Delta=\Delta_x,\ldots $.
	\item In contrast, $ C^k $-norms are always evaluated on the space-time $ [0,T]\times\Td $.
	\item If a function is stated to have zero mean value we always mean `in the space variable'. Define $ C_0^\infty $ to be the space of smooth functions which have zero mean value:
	\[ C_0^\infty(\Td) \coloneqq \left\{f:\Td\to\R \text{ smooth such that } \fint_{\Td}f(x)dx =0 \right\}. \]
\item If not specified otherwise, for a periodic function $ f:\Td\to\R $ and  $ \lambda \in \N_+$, $ f_\lambda:\Td\to\R $ denotes the dilation $ f_\lambda(x) = f(\lambda x) $. Note that 
	\begin{equation}
	\label{eq:osc}
	\|D^k f_\lambda\|_{L^p(\Td)} = \lambda^k \|D^k f\|_{L^p(\Td)}. 
	\end{equation}
\end{itemize}

\subsection*{Acknowledgment}
This research was supported by the ERC Grant Agreement No.~724298. The authors wish to thank Prof. L\'aszl\'o Sz\'ekelyhidi and Jonas Hirsch for several useful discussions on the topic of this paper.

\section{Main Proposition and proof of the theorem}
\label{sec:prop}

In this section we state the main proposition of this paper, Proposition \ref{prop:main}, and we use it in order to prove Theorem \ref{thm:weak}. Proposition \ref{prop:main} will be proven in details in Sections \ref{sec:tools}-\ref{sec:proof}, assuming, for simplicity, $p>1$. A sketch of the proof in the case $p=1$ can be found in Section \ref{subsec:continuous-vf}. 

  We introduce the (incompressible) continuity-defect equation
\begin{equation}
	\left.
	\begin{aligned}
		\partial_t \rho + \div(\rho u) &= -\div R \\
		\div u &=0
	\end{aligned}
	\right\}
	\label{eq:cont-defect}
\end{equation}
as an approximation of the transport equation. The iteration step of the Convex Integration scheme deals with solution to this system.

\begin{prop}
\label{prop:main}
	There is a constant $M>0$ such that the following holds. Let $p\in[1,\infty)$ and $\tilde{p}\in[1,\infty)$ so that
	\begin{equation}
	\frac{1}{p} +\frac{1}{\tilde{p}} > 1+\frac{1}{d}.
	\label{eq:p-cond}
	\end{equation}
	Then for any $\delta,\eta>0$ and any smooth solution $(\rho_0,u_0,R_0)$ of the continuity defect equation \eqref{eq:cont-defect} there is another smooth solution $(\rho_1,u_1,R_1)$ which fulfils the estimates
	\begin{subequations}
		\begin{align}
		\|\rho_1(t)-\rho_0(t)\|_{L^p} &\le M\eta \|R_0(t)\|_{L^1}^{1/p}
		\label{eq:rho1-rho0_Lp} \\
		\|u_1(t)-u_0(t)\|_{L^{p'}} &\le \frac{M}{\eta} \|R_0(t)\|_{L^1}^{1/p'}
		\label{eq:u1-u0_Lp'} \\
		\|u_1(t)-u_0(t)\|_{W^{1,\tilde{p}}} &\le \delta
		\label{eq:u1-u0_W1p} \\
		\|R_1(t)\|_{L^1} &\le \delta
		\label{eq:R1}
		\end{align}
	\end{subequations}
	for all $ t\in[0,T] $.
	Furthermore the solution is not changed at times where it is a proper solution of \eqref{eq:continuity-eqn}--\eqref{eq:incompressibility}, i.e.~if $R_0(t,\cdot)\equiv 0$ for some $t\in[0,T]$ then $ R_1(t)\equiv 0 $ and $(\rho_1,u_1)(t)\equiv(\rho_0,u_0)(t)$.
\end{prop}

\begin{proof}[Proof of Theorem~\ref{thm:weak}, assuming Proposition \ref{prop:main}]
We will use the proposition to construct a sequence $ (\rho_n,u_n,R_n)_{n\in\N} $ of solutions to \eqref{eq:cont-defect} in the space $$ C\left([0,T], L^p(\Td)\times \left( L^{p'}\cap W^{1,q-}(\Td,\R^d)\right) \times L^1(\Td,\R^d) \right)$$ (with $q$ as in \eqref{eq:p-q-cond}), which in the limit will produce a solution of \eqref{eq:continuity-eqn}--\eqref{eq:incompressibility}.

Set $ (\rho_0,u_0) \coloneqq (\bar{\rho},\bar{u})$ as given in the statement of the theorem and define
\[ R_0(t)\coloneqq -\nabla\Delta^{-1} \left[ \partial_t \bar{\rho}(t) + \div\left( \bar{\rho}(t)\bar{u}(t) \right) \right]. \]
Recall that $ \partial_t \bar{\rho} $ has zero mean value by assumption and $ \div(\bar{\rho}\bar{u}) $ also, being a divergence, so the definition is correct. Then clearly $ (\rho_0,u_0,R_0) $ is a smooth solution of \eqref{eq:cont-defect}.

Set $ \delta_0\coloneqq\|R_0\|_{C_tL_x^1} $ and choose a sequence of positive numbers $ \delta_n $, $ n\ge1 $ such that the sum $\sum_n \delta_n^{1/2}$ converges. (Then in particular $\sum_n \delta_n <\infty$.) Furthermore choose sequences $ (\tilde{p}_n)_{n\in\N} \subset [1,q)$ and $ (\eta_n)_{n\in\N} \subset (1,\infty) $ such that
\[ \tilde{p}_n\xrightarrow{n\to\infty} q \ \text{ and }\  \delta_n^{1/p}\eta_n=\sigma\delta_n^{1/2} \]
for some $ \sigma>0 $ to be chosen later and observe that $ \delta_n^{1/p'}/\eta_n= \delta_n^{1/2}/\sigma $. By repeated application of Proposition~\ref{prop:main} we obtain a sequence of smooth solutions $ (\rho_n,u_n,R_n) $ fulfilling the bounds (uniformly in time)
\begin{subequations}
	\begin{align}
		\|\rho_{n+1}(t)-\rho_n(t)\|_{L^p} \le M\eta_n \|R_n(t)\|_{L^1}^{1/p} &\le M \sigma\delta_n^{1/2}
		\label{eq:rhon+1-rhon_Lp} \\
		\|u_{n+1}(t)-u_n(t)\|_{L^{p'}} \le \frac{M}{\eta_n} \|R_n(t)\|_{L^1}^{1/p'} &\le \frac{M}{\sigma} \delta_n^{1/2}
		\label{eq:un+1-un_Lp'} \\
		\|u_{n+1}(t)-u_n(t)\|_{W^{1,\tilde{p}_n}} &\le \delta_{n+1}
		\label{eq:un+1-un_W1p} \\
		\|R_{n+1}(t)\|_{L^1} &\le \delta_{n+1}
		\label{eq:Rn+1} \\
		R_n(t) = 0 \implies R_{n+1}(t) &= 0.
		\label{eq:Rn(t)0}
	\end{align}
\end{subequations}
Clearly there are functions $ \rho\in C_tL^p_x $ and $ u\in C_tL^{p'}_x\cap C_tW^{1,\tilde{p}}_x $ for any $ \tilde{p}<q $ such that
$ \rho_n \to \rho $ in $ C_tL^p_x $ and $ u_n \to u $ in $ C_tL^{p'}_x $ and $ C_tW^{1,\tilde{p}}_x $. Moreover, we have $ \rho_n u_n \to \rho u$ and $ R_n \to0 $ in $ C_tL^1_x $, which proves statements (i) and (ii) of the theorem. For $ t\in E $ by \eqref{eq:Rn(t)0} we have $ R_n(t) =0$ for all $n$ and therefore, by \eqref{eq:rhon+1-rhon_Lp} and \eqref{eq:un+1-un_Lp'}
\[ \rho_n(t)=\bar{\rho}(t),\ u_n(t)=\bar{u}(t)\ \forall\,n \]
which implies statement (iii). For the last statement we need to choose a sufficiently small (or large) $\sigma$ so that $ M\sigma\sum_{n=0}^\infty \delta_n^{1/2}<\varepsilon$ (or $ M\sigma^{-1}\sum_{n=0}^\infty \delta_n^{1/2}<\varepsilon $). So we can ensure that statement (iv) (or statement (iv)', respectively) holds by our choice of $ \sigma $. If $p=1$ (and thus $p'=\infty$), then the continuity in space-time of the limit $u$ follows from \eqref{eq:un+1-un_Lp'}, observing  that, in this case, $u$ is the uniform limit of the smooth vector fields $u_n$. This concludes the proof of the main theorem.
\end{proof}

We will only prove Proposition~\ref{prop:higherReg} in the case $ p>1 $, the proof will cover \Cref{sec:error,sec:perturbation,sec:proof}. The case $ p=1 $, in which the obtained velocity field is in particular continuous (although continuity via Sobolev embeddings just exactly fails to hold), is more delicate to prove. We refer to~\cite{modena-szekelyhidi18} for the details and will sketch the strategy and the necessary adaptations in \Cref{sec:appendix}.

\section{Technical Tools}
\label{sec:tools}

In this section we provide some technical tools we will use throughout the paper.

\subsection{Improved Hölder inequality for fast oscillations}

We recall the following lemma from \cite{Modena2018}:
\begin{lem}
	\label{lem:improvedH}
	For $ p\in[1,\infty] $ there is a constant $ C_p $ such that for all smooth functions $ f,g $ on the torus $ \Td $ and $ \lambda\in\N $:
	\begin{equation*}
		\big|\| fg_\lambda \|_{L^p} -\|f\|_{L^p} \|g\|_{L^p} \big| \le \frac{C_p}{\lambda^{1/p}} \|f\|_{C^1} \|g\|_{L^p}.
		\label{eq:Holder}
	\end{equation*}
\end{lem}
\begin{rem}
	In particular this lemma supplies the Hölder-like inequality
	\begin{equation}
	\|fg_\lambda\|_{L^p} \le \|f\|_{L^p} \|g\|_{L^p} + \frac{C_p}{\lambda^{1/p}} \|f\|_{C^1} \|g\|_{L^p}.
	\label{eq:improvedH}
	\end{equation}
	which allows to bound the product by the $L^p$ norm of both functions, plus some error term which is small if one function is fastly oscillating, i.e.~$ \lambda $ is large.  
\end{rem}

\subsection{Higher Derivatives and Antiderivatives}
As for smooth $ f $, with $\fint_{\Td} f = 0$, the Poisson equation $ \Delta u=f $ has a solution on the flat torus which is unique up to addition of a constant, the inverse Laplacian
\[ \Delta^{-1} : C_0^\infty\to C_0^\infty ,\ f\mapsto u \]
is well-defined as an operator on the space $ C_0^\infty $. We can now use it to define higher order (anti)derivatives with a simple structure.

\begin{defi}
	For any smooth function $f \in C^\infty(\Td)$ on the torus and non-negative integers $k$ we define the differential operator $ \D^k $:
\[ \D^k f  =
\begin{dcases*}
\Delta^{k/2} f, & if $k$ even, \\
\nabla\Delta^{\frac{k-1}{2}} f, & if $k$ odd,
\end{dcases*} \]
with the convention that $ \D^0=\Delta^0=Id $.

For negative $k$ the definition is identical with the additional condition $ f\in C_0^\infty(\Td) $, which is necessary so that negative powers of the Laplacian are meaningful.
\end{defi}

\begin{rem} 
	The basic properties of the operators $ \D^k $ include
	\begin{itemize}
		\item Commutes with derivatives: $ \partial^\alpha \D^k f = \D^k\partial^\alpha f $ for all $ k\in\Z $ and any multi-index $ \alpha $.
		\item Partial Integration: For any $ k,n,m\in\Z $ and $ f,g\in C_0^\infty(\Td) $
		\[ \int_{\Td} \D^kf \cdot \D^{m+n}g = (-1)^n \int_{\Td} \D^{k+n}f \cdot \D^mg, \]
		where the `$ \cdot $' denotes scalar product if both factors are vectors, otherwise standard multiplication.
		\item  Scaling: $ \D^ku_\lambda = \lambda^k(\D^k u)_\lambda $ for any  $ k\in\Z $ and $ \lambda\in\N $.
	\end{itemize}
\end{rem}

\subsection{Calderon-Zygmund estimates}

We first recall the usual Calderon-Zygmund inequality in the following form.
\begin{rem}[Classical Carderon-Zygmund inequality]
	Let $ p\in (1,\infty) $. There is a constant $ C_{d,p}$ such that for any smooth compactly supported function $ f $ the following inequality holds:
	\begin{equation}\label{eq:classic-C-Z}
	\|f\|_{W^{2,p}(\R^d)} \le C_{d,p} \|\Delta f\|_{L^p(\R^d)}.
	\end{equation}
We refer to \cite{gilbart01} for the proof. 
\end{rem}

It is now a small step to show that the same statement can be transferred to the periodic setting: we include the proof for completeness. 

\begin{lem}[Calderon-Zygmund on the flat torus]
	Let $ p\in(1,\infty) $. There is a constant $ C_{d,p}$ such that for any $ f\in C_0^\infty(\Td) $ the following inequality holds:
	\begin{equation}\label{eq:C-Z-torus}
	\|f\|_{W^{2,p}(\Td)} \le C_{d,p} \|\Delta f\|_{L^p(\Td)}.
	\end{equation}
	\label{lem:C-Z-classic}
\end{lem}

\begin{proof}
	Let $ f\in C_0^\infty(\Td) $ and $ N\in\N $. We treat $f$ as a periodic map $f: \R^d \to \R$ and identify $\Td$ with the unit cube $(0,1)^d$. Choose a smooth cut-off function $ \chi\in C^\infty(\R) $ such that $ \chi(x)=1 $ if $ x\le0 $ and $ \chi(x)=0 $ if $ x\ge1 $. Define the function $ f_N\in C_c^\infty(\R^d) $ by
	\[ f_N(x) \coloneqq \left(\prod_{i=1}^d \chi(|x_i|-N)\right) f(x). \]
	Now the classical Calderon-Zygmund inequality \eqref{eq:classic-C-Z} and the fact that $ f_N $ is supported in the cube $ [-N-1,N+1]^d $ yield
	\[ \|f_N\|_{W^{2,p}([-N,N]^d)} \le \|f_N\|_{W^{2,p}(\R^d)} \le C_{d,p} \|\Delta f_N\|_{L^p(\R^d)} = C_{d,p} \|\Delta f_N\|_{L^p([-N-1,N+1]^d)} \]
	and therefore, using that $ \|\chi\|_{C^0}=1 $ and $ f_N = f $ on $ [-N, N]^d $. 
	\begin{align*}
		(2N)^d \|f\|_{W^{2,p}(\Td)} \le& C_{d,p} \left[ (2N+2)^d \|\Delta f\|_{L^p(\Td)} \right.\\
		&\left. + \left((2N+2)^d-(2N)^d\right) \left(\|\chi'\|_{C^0}\|\nabla f\|_{L^p(\Td)} + \|\chi''\|_{C^0}\|f\|_{L^p(\Td)} \right)\right].
	\end{align*}
	If $ N\to\infty $ the dominating terms are the ones with the factor $ (2N)^d $, and so
	\[ \|f\|_{W^{2,p}(\Td)} \le C_{d,p}\|\Delta f\|_{L^p}(\Td) \]
	holds with the same constant as in the full space setting.
\end{proof}

\begin{lem}[Estimates on antiderivatives]
	Let $ p\in(1,\infty) $ and $ k\in\N $. There is a constant $ C_{d,p,k}$ such that
	\begin{equation}\label{eq:antiderivative-est}
	\|\D^{-k} f\|_{W^{k,p}(\Td)} \le C_{d,p,k} \|f\|_{L^p(\Td)}
	\end{equation}
	holds for any $ f\in C_0^\infty(\Td) $.
	\label{lem:antiderivative-est}
\end{lem}

\begin{proof}
	If $ k $ is even, the inequality arises simply from iterated application of the Calderon-Zygmund inequality on the torus:
	\[ \|\D^{-k}f\|_{W^{k,p}} = \|\Delta^{-k/2} f\|_{W^{k,p}} \le C_{d,p} \|\Delta^{-k/2+1}f\|_{W^{k-2,p}} \le \ldots \le C_{d,p}^{k/2} \|f\|_{L^p}. \]
	For odd numbers $ k $ observe that the same iteration leaves us with
	\[ \|\D^{-k}f\|_{W^{k,p}} \le C_{d,p}^{(k-1)/2} \|\D^{-1}f\|_{W^{1,p}} = C_{d,p}^{(k-1)/2} \|\nabla\Delta^{-1}f\|_{W^{1,p}} \]
	and clearly
	\[ \|\nabla\Delta^{-1}f\|_{W^{1,p}} \le \|\Delta^{-1}f\|_{W^{2,p}} \le C_{d,p} \|f\|_{L^p} \]
	so the stated inequality holds with $ C_{d,p,k} = C_{p,d}^{\lceil  k/2 \rceil} $.
\end{proof}

\begin{lem}[End point estimates on antiderivatives]
	Let $ p\in[1,\infty] $ and $ k\in\N_+ $. There is a constant $ C_{d,p,k} $ such that
	\begin{equation}\label{eq:antiderivative-end}
	\|\D^{-k} f\|_{W^{k-1,p}(\Td)} \le C_{d,p,k}\|f\|_{L^p(\Td)}
	\end{equation}
	holds for any $ f\in C_0^\infty(\Td) $.
	\label{lem:antiderivative-end}
\end{lem}

\begin{proof}
	In the case $ p\in(1,\infty) $ there is nothing to show as the statement is just a weaker form of \eqref{eq:antiderivative-est}.
	
	For $ p=\infty $ we use Sobolev embeddings on every derivative of order $ k-1 $ and smaller to control the Sobolev norm of a smooth function $ g $: for every multiindex $\alpha$, with $|\alpha| \leq k-1$, 
	\[
		\|\partial^\alpha g\|_{L^\infty} \le C_{d} \|D\partial^\alpha g\|_{L^{d+1}} 
		\implies \|g\|_{W^{k-1,\infty}} \le C_{d} \|g\|_{W^{k,d+1}}.
	\]
	If we set $ g=\D^{-k}f $ and we use the previous Lemma, we obtain
	\[ \|\D^{-k}f\|_{W^{k-1,\infty}} \le C_{d} \|\D^{-k}f\|_{W^{k,d+1}} \le C_{d,p,k} \|f\|_{L^{d+1}} \le C_{d,p,k} \|f\|_{L^\infty}. \]
	
	For $ p=1 $ we consider the dual characterisation of the $ L^1 $-norm:
		\begin{align*}
			\|g\|_{L^1} &= \max\left\lbrace \frac{1}{\|\phi\|_{L^\infty}} \int_{\Td}g\phi : \phi\in L^\infty(\Td) \setminus\{0\} \right\rbrace \\
			&= \sup\left\lbrace \frac{1}{\|\phi\|_{L^\infty}} \int_{\Td}g\phi : \phi\in C^\infty(\Td) \setminus\{0\} \right\rbrace.
		\end{align*}
		If $ \fint_{\Td}g =0$ we can restrict the definition to test functions in $ C_0^\infty(\Td) $, still obtaining the inequalities
		\begin{equation}
		\frac{1}{2}\|g\|_{L^1} \le  \sup\left\lbrace \frac{1}{\|\phi\|_{L^\infty}} \int_{\Td}g\phi : \phi\in C_0^\infty(\Td) \setminus\{0\} \right\rbrace \le \|g\|_{L^1}
		\end{equation}
		where the first inequality comes from the fact that $\int g(\phi-\fint\phi) = \int g\phi$ and \linebreak[4]$ \|\phi-\fint\phi\|_{L^\infty} \le 2 \|\phi\|_{L^\infty} $ hold for any $ \phi $.  Using this, we can estimate for any multiindex $ \alpha $ of order $ k-1 $ or smaller
		\begin{align*}
			\|\partial^\alpha \D^{-k}f\|_{L^1} &\le \sup_{\phi\in C_0^\infty(\Td)} \frac{2}{\|\phi\|_{L^\infty}} \int_{\Td} \left(\partial^\alpha \D^{-k}f\right) \phi \\
			&= \sup_{\phi\in C_0^\infty(\Td)} \frac{2}{\|\phi\|_{L^\infty}} \int_{\Td} f \left(\partial^\alpha \D^{-k} \phi\right) \\
			&\le \sup_{\phi\in C_0^\infty(\Td)} \frac{2}{\|\phi\|_{L^\infty}} \|f\|_{L^1} \|\partial^\alpha \D^{-k} \phi\|_{L^\infty} \\
			&\le \|f\|_{L^1} \sup_{\phi\in C_0^\infty(\Td)} \frac{C_{d,\infty, k}}{\|\phi\|_{L^\infty}}  \|\phi\|_{L^\infty} \\
			&= C_{d,\infty,k} \|f\|_{L^1}
		\end{align*}
		where in the last inequality \eqref{eq:antiderivative-end} with $ p=\infty $ was applied. Summation over all such $ \alpha $ then yields \eqref{eq:antiderivative-end}:
		\[ \|\D^{-k}f\|_{W^{k-1,1}} = \sum_{\mathclap{|\alpha|\le k-1}} \|\partial^\alpha \D^{-k}f\|_{L^1} \le \sum_{\mathclap{|\alpha|\le k-1}} C_{d,\infty,k} \|f\|_{L^1} = C_{d,1,k} \|f\|_{L^1}. \qedhere \]
\end{proof}

\subsection{Improved antidivergence for fast oscillations}
\label{subsec:adiv}

	The first order antiderivative $ \D^{-1} $ is an antidivergence operator, which we will call standard antidivergence operator. It will be used in situations when the estimate provided in Lemma~\ref{lem:antiderivative-end} with $ k=1 $ suffices.
However, in many steps of the proof of Proposition~\ref{prop:main} refined estimates on the anitdivergence are necessary. We therefore introduce a bilinear operator which is apt to control the anitdivergence of a product of functions if one of them is fastly oscillating.

\begin{defi}[Bilinear antidivergence operator]
	Let $ N\in\N $. Define the operator
	\begin{equation}\label{eq:def-antidiv}
	\begin{gathered}
	\adiv_N : C^\infty(\Td) \times C_0^\infty(\Td) \to C^\infty(\Td;\R^d) \\
	\adiv_N \left(f,g\right) \coloneqq \sum_{k=0}^{N-1} (-1)^k \D^k f \D^{-k-1} g + \D^{-1} \left( (-1)^N \D^N f \cdot\D^{-N} g - \fint_{\Td}fg\right).
	\end{gathered}
	\end{equation}
	Here the `$ \cdot $' indicates the scalar product if needed, i.e.~if $ N $ is odd, and the standard product otherwise.
	Note that both arguments must be smooth but only the second argument $ g $ is supposed to have zero mean value.
\end{defi}

\begin{lem}[Properties of $ \adiv_N $]
	\label{lem:adiv-properties}
	Let $ N\in\N $, $ f\in C^\infty(\Td) $ and $ g\in C_0^\infty(\Td) $.
	\begin{enumerate}[(i)]
		\item $ \adiv_N $ is an anitdivergence operator in the sense that
		\[ \div \left(\adiv_N (f,g)\right) = fg-\fint_{\Td}fg. \]
		\item $ \adiv_N $ satisfies the Leibniz rule:
		\[ \partial_j \left(\adiv_N(f,g)\right) = \adiv_N(\partial_jf,g) + \adiv_N(f,\partial_jg). \]
		\item If $ p,r,s\in[1,\infty] $ such that $ \frac{1}{p}=\frac{1}{r}+\frac{1}{s} $, then the following inequality holds:
		\begin{equation}\label{eq:adiv-raw-est}
		\left\| \adiv_N(f,g) \right\|_{L^p} \le \sum_{k=0}^{N-1} \|\D^kf\|_{L^r} \|\D^{-k-1}g\|_{L^s} + C_{d,p}\|\D^Nf\|_{L^r} \|\D^{-N}g\|_{L^s}.
		\end{equation}
	\end{enumerate}
\end{lem}

\begin{proof}
	(i) By induction in $ N $. By definition we have
	\[ \adiv_0 (f,g) = \D^{-1} \left(fg-\fint_{\Td}fg\right) \]
	so the statement follows from the remark on standard antidivergence.
	Now let $ N>0 $ and w.l.o.g assume $ N $ to be even, then
	\begin{align*}
		\div \left(\adiv_N (f,g)\right) - \left(fg-\fint_{\Td}fg\right) &=
		\overbrace{-\left(fg-\fint_{\Td}fg\right) + \div \left(\adiv_{N-1} (f,g)\right)}^{=0 \text{ by assumption}} \\
		&\quad-(-1)^{N-1} \D^{N-1}f\cdot\D^{-N+1}g + (-1)^N \D^Nf \D^{-N}g \\
		&\quad+ \div \left((-1)^{N-1} \D^{N-1}f\D^{-N}g\right) \\
		&= \D^{N-1}f\cdot\D^{-N+1}g + \D^Nf \D^{-N}g \\
		&\quad- \div\left(\D^{N-1}f\right)\D^{-N}g - \D^{N-1}f \cdot \nabla\D^{-N}g \\
		&= 0
	\end{align*}
	by definition of the operators $ \D^k $.
	
	(ii) is proven by lengthy but straightforward computation which we omit here.
	
	(iii) Use the standard Hölder inequality on each term of the definition of $ \adiv_N $. For the last summand note that Lemma~\ref{lem:antiderivative-end} in particular implies $ \|\D^{-1}h\|_{L^p} \le C(d,p) \|h\|_{L^p} $; furthermore $ \|h-\fint_{\Td}h\|_{L^p} \le 2\|h\|_{L^p}$ for any $ p $.
\end{proof}

\begin{rem}
	The bilinear antidivergence and inequality \eqref{eq:adiv-raw-est} are only useful if applied on fuctions $ g_\lambda $ which are fast oscillating, as then we gain the oscillation parameter $ \lambda $ as small factor. In particular the following two estimates will be used throughout the paper.
	Let $ p\in[1,\infty] $, $ \lambda,N\in\N $, $ f\in C^\infty(\Td) $ and $ g\in C_0^\infty(\Td) $. Then:
	\begin{align}
		\left\| \adiv_N(f,g_\lambda) \right\|_{L^p} &\le C_{d,p,N} \|g\|_{L^p} \left( \sum_{k=0}^{N-1} \lambda^{-k-1} \|\D^k f\|_{L^\infty} + \lambda^{-N} \|\D^Nf\|_{L^\infty}\right),
		\label{eq:antidiv-in-infty} \\
		\left\| \adiv_N(f,g_\lambda) \right\|_{L^p} &\le C_{d,p,N} \|g\|_{L^\infty} \left( \sum_{k=0}^{N-1} \lambda^{-k-1} \|\D^k f\|_{L^p} + \lambda^{-N} \|\D^Nf\|_{L^p} \right).
		\label{eq:antidiv-in-p}
	\end{align}
	The proof of \eqref{eq:antidiv-in-infty}-\eqref{eq:antidiv-in-p} is direct consequence of \eqref{eq:adiv-raw-est} and Lemma~\ref{lem:antiderivative-end}.
\end{rem}

\section{The perturbations}
\label{sec:perturbation}

In this section we introduce the basic building blocks of our construction, namely the \emph{space-time Mikado densities and field}, which allow us to get a ``full dimensional concentration'', i.e. to assume \eqref{eq:exponents-main-condition} instead of \eqref{eq:exponents-ms}. We then use the Mikado functions to define and estimate  $\rho_1, u_1$. 

\subsection{Space-time Mikado densities and fields}
\label{subsec:mikado}
%
%
%
%

For given $\zeta, v \in \Td$, consider the line on $\Td$
\begin{equation*}
\R \ni s \mapsto \zeta + s v \in \Td.
\end{equation*}

\begin{lem}[Space-time Mikado lines]
\label{lem:mikado-lines}
There exist $r>0$ and $\zeta_1, \dots, \zeta_d \in \Td$ such that the lines
	\[ \mathtt x_j: \R \to\Td, \ \mathtt x_j(s) = \zeta_j + s e_j \]
satisfy
\begin{equation}
\label{eq:lines-disjoint}
d_{\Td}(\mathtt x_i(s), \mathtt x_j(s)) > 2r \quad \forall s\in \R, \  \forall i\ne j,
\end{equation}
	where $ d_{\Td} $ denotes the Euclidian distance on the torus.
\end{lem}

\begin{rem}
We can think to the lines $\mathtt x_j$ as the trajectories of $d$ particles moving on the torus with speed $1$ and along different directions. The claim of the Lemma is that such particles have different positions at every time. 
\end{rem}

\begin{proof}
We define
\begin{equation*}
\zeta_i := \frac{i}{d} e_i, \qquad i=1, \dots, d.
\end{equation*}

Let $i \neq j$ be fixed. If, for some  $s \in \R$,
%
%
%
\begin{equation*}
\mathtt x_i(s) = \mathtt x_j(s) \text{ in } \Td,
\end{equation*}
then
\begin{equation*}
(\zeta_j + s e_j) - (\zeta_i + s e_i)  \in \Z^d
\end{equation*}
and thus
\begin{equation*}
\frac{i}{d} + s \in \Z, \qquad \frac{j}{d} + s \in \Z,
\end{equation*}
which implies, taking the difference,
\begin{equation*}
\frac{i - j}{d} \in \Z,
\end{equation*}
a contradiction. Therefore, for every $s \in \R$ and $i \neq j$, $\mathtt x_i(s) \neq \mathtt x_j(s)$ and thus there must be $r>0$ such that \eqref{eq:lines-disjoint} holds. 
\end{proof}

Let $\varphi$ be a smooth function on $\R^d$, with
\begin{equation*}
\mathrm{supp} \  \varphi \subseteq B(P, r) \subseteq (0,1)^d,
\end{equation*} 
where $P = (1/2, \dots, 1/2) \in (0,1)^d$, and so that $ \varphi$ fulfill
\[
\int_{\R^d} \varphi^2 = 1. 
\]
%
%
%
%
%
%
For a given $p$ (fixed in the statement of Proposition \ref{prop:main}),  and its dual Hölder exponent $p'$ define the constants
\begin{equation}
a\coloneqq \frac{d}{p} ,\ b\coloneqq \frac{d}{p'} \ \text{ so that } a+b=d
\label{eq:def-a-b}
\end{equation}
and the scaled functions (defined on the whole space $\R^d$, thus not periodic)
\[\varphi_\mu(x) \coloneqq \mu^a \varphi(\mu x),\ \tilde{\varphi}_\mu(x) \coloneqq \mu^b \varphi(\mu x), \quad \mu \geq 1.  \]

\begin{lem}
For every $\mu \geq 1$, $k \in \N$, $r \in [1, \infty]$,
\begin{equation}\label{eq:phi_mu-scaling}
\|D^k\varphi_\mu\|_{L^r} =  \mu^{a-\frac{d}{r}+k} \|D^k \phi\|_{L^r} \ ,\ \|D^k\tilde{\varphi}_\mu\|_{L^r} = \mu^{b-\frac{d}{r}+k}\|D^k \phi\|_{L^r} . 
\end{equation}
Moreover,
\begin{equation}
\label{eq:mean-value-1}
\int_{\R^d} \varphi_\mu \tilde \varphi_\mu = 1. 
\end{equation}
\end{lem}
\noindent The proof is straightforward and thus it is omitted. Note in particular that the $L^p$-norm of $\varphi_\mu$ and the $L^{p'}$-norm of $\tilde{\varphi}_\mu$ are invariant of the scaling. Note also that $\mathrm{supp} \, \varphi_\mu = \mathrm{supp} \, \tilde \varphi_\mu$ and both are contained in a ball with radius at most $r$.  For any given $y \in \Td$, we  define the translation
\begin{equation*}
\tau_y : \Td \to \Td, \qquad \tau_y(x) := x - y.
\end{equation*}
Notice that, for every smooth periodic map $g$
\begin{equation*}
\|D^k (g \circ \tau_y)\|_{L^r} = \|D^k g\|_{L^r} \quad \forall k \in \N, \ \forall r \in [1, \infty].
\end{equation*} 

\begin{lem}
There are periodic functions 
\begin{equation*}
\varphi_\mu^j : \Td \to \R, \qquad \tilde \varphi_\mu^j : \Td \to \R, \qquad j=1,\dots, d,
\end{equation*}
such that the same scaling as in \eqref{eq:phi_mu-scaling} holds:
\begin{equation}
\label{eq:phi-j-mu-scaling}
\|D^k\varphi_\mu^j\|_{L^r} =  \mu^{a-\frac{d}{r}+k} \|D^k \varphi\|_{L^r} \ ,\ \|D^k\tilde{\varphi}_\mu^j\|_{L^r} = \mu^{b-\frac{d}{r}+k}\|D^k \varphi\|_{L^r}.
\end{equation}
Moreover, for every $i=1, \dots, d$, 
\begin{equation}
\label{eq:mean-value-1-j}
\fint_{\Td} \big(\varphi_\mu^i  \circ \tau_{se_i}\big)\big( \tilde \varphi_\mu^i \circ \tau_{se_i}\big) = 1,
\end{equation}
and, for every $i \neq j$ and $s \in \R$,
\begin{equation}
\label{eq:disjoint-support}
\left( \varphi_\mu^i \circ \tau_{s e_i} \right) \left( \tilde \varphi_\mu^j \circ \tau_{s e_j} \right) = 0.
\end{equation}
\end{lem}
\noindent Notice that \eqref{eq:disjoint-support} means
\begin{equation*}
\varphi_\mu^i (x - s e_i) \tilde \varphi^j (x - s e_j) = 0
\end{equation*}
for every $x \in \Td$. 
\begin{proof}
Since $\varphi_\mu$, $\tilde \varphi_\mu$ have support  contained in $(0,1)^d$, we can consider their periodic extensions, still denoted, with a slight abuse of notation, by $\varphi_\mu$, $\tilde \varphi_\mu$, respectively. We define now the periodic maps
\begin{equation*}
\varphi_\mu^j : = \varphi_\mu \circ \tau_{\zeta_j}, \qquad \tilde \varphi_\mu^j := \tilde \varphi_\mu^j \circ \tau_{\zeta_j},
\end{equation*}
where $\zeta_1, \dots, \zeta_d$ are the points given by Lemma \ref{lem:mikado-lines}. 
It is immediate from the definition and from \eqref{eq:phi_mu-scaling}-\eqref{eq:mean-value-1} that \eqref{eq:phi-j-mu-scaling}-\eqref{eq:mean-value-1-j} holds. Let now $x \in \Td$, $s \in \R$. We have
\begin{equation*}
\varphi_\mu^i(x - s e_i) \tilde \varphi_\mu^j(x - se_j) = \varphi_\mu(x - \zeta_i - s e_i) \tilde \varphi_\mu(x - \zeta_j - se_i) = \varphi_\mu(x -  \mathtt x_i(s)) \tilde \varphi_\mu(x -  \mathtt x_j(s)).
\end{equation*}
Observe that, by Lemma \ref{lem:mikado-lines}, 
\begin{equation*}
d_{\Td}\Big(x - \mathtt x_i(s), \, x - \mathtt x_j(x) \Big) = d_{\Td} \left( \mathtt x_i(s), \mathtt x_j(s) \right) > 2r.
\end{equation*}
Since the support of $\varphi_\mu$ and $\tilde \varphi_\mu$ coincide and are both contained in a ball with radius at most $r$, it must be
\begin{equation*}
\varphi_\mu(x -  \mathtt x_i(s)) \tilde \varphi_\mu(x -  \mathtt x_j(s)) = 0,
\end{equation*}
and thus \eqref{eq:disjoint-support} holds. 
\end{proof}

We introduce now the building block of our construction, the space-time Mikado densities and fields. Besides the families of functions $\varphi_\mu^j$, $\tilde \varphi_\mu^j$, $\mu \geq 1$, $j=1,\dots, d$, we fix a smooth periodic function $\psi: \mathbb{T}^{d-1} \to \R$ satisfying
\begin{equation*}
\fint_{\mathbb{T}^{d-1}} \psi = 0, \qquad \fint_{\mathbb{T}^{d-1}} \psi^2 = 1
\end{equation*}
and we define
\begin{equation*}
\psi^j : \Td \to \R, \qquad \psi^j(x) = \psi^j(x_1, \dots, x_d) := \psi(x_1, \dots, x_{j-1}, x_{j+1}, \dots, x_d),
\end{equation*}
for every $j=1,\dots, d$, so that
\begin{equation}
\label{eq:psi:mean:value}
\fint_{\mathbb{T}^{d-1}} \psi^j = 0, \qquad \fint_{\mathbb{T}^{d-1}} (\psi^j)^2 = 1.
\end{equation}
Introduce the parameters
\[ \begin{array}{rll}
\lambda & \text{`fast oscillation',} &\in\N \\
\mu & \text{`concentration',} &\gg \lambda \\
\omega & \text{`phase speed'} & \\
\nu & \text{`very fast oscillation',} & \in\lambda\N,\ \gg\lambda
\end{array} \]
to be chosen in the very end of the proof.
Now we can define the Mikado functions, for $j=1,\dots, d$:
\begin{align*}
&\text{Mikado density} & \Theta^j_{\lambda,\mu,\omega,\nu}(t,x) &\coloneqq \varphi_\mu^j \left( \lambda (x-\omega t e_j) \right) \psi^j(\nu x), \\
&\text{Mikado field} & W^j_{\lambda,\mu,\omega,\nu}(t,x) &\coloneqq \tilde{\varphi}_\mu^j \left( \lambda (x-\omega t e_j) \right) \psi^j (\nu x)e_j, \\
&\text{Quadratic corrector} & Q_{\lambda,\mu,\omega,\nu}^j (t,x) &\coloneqq \frac{1}{\omega} \left(\varphi_\mu^j \tilde{\varphi}_\mu^j\right) \left(\lambda(x-\omega t e_j)\right) \left(\psi^j(\nu x)\right)^2.
\end{align*}
We will use also the shorter notation
\begin{align*}
\Theta^j_{\lambda,\mu,\omega,\nu} = \Theta^j_{\lambda,\mu,\omega,\nu}(t) & \coloneqq \left( (\varphi_\mu^j)_\lambda \circ \tau_{\omega t e_j} \right) \psi^j_\nu, \\
W^j_{\lambda,\mu,\omega,\nu} = W^j_{\lambda,\mu,\omega,\nu}(t) &\coloneqq \left(  (\tilde{\varphi}_\mu^j )_\lambda \circ \tau_{\omega t e_j} \right)  \psi^j_\nu e_j,  \\
Q_{\lambda,\mu,\omega,\nu}^j = Q_{\lambda,\mu,\omega,\nu}^j (t) &\coloneqq \frac{1}{\omega} \left( \left(\varphi_\mu^j \tilde{\varphi}_\mu^j\right)_\lambda \circ \tau_{\omega t e_j} \right) \left(\psi^j_\nu \right)^2,
\end{align*}
where we have used the notation $g_\lambda(x) := g(\lambda x)$ (and $g_\nu(x) := g(\nu x)$), for $g: \Td \to \R$.

\begin{rem}
	The Mikados defined here do not form a stationary solution of the incompressible transport equation, in contrast to those used in \cite{Modena2018,modena-szekelyhidi18}. The ideal cancellation properties
	$\partial_t \, \Theta^j_{\lambda, \mu, \omega, \nu}= \div(\Theta^j_{\lambda, \mu, \omega, \nu} W^j_{\lambda, \mu, \omega, \nu} 	) = 0 = \div W^j_{\lambda, \mu, \omega, \nu}$
		cannot hold here because of the time-dependence and compact support in space of the function $ \varphi(\lambda(x-\omega te_j)) $. However, $ \psi $ is still time-independent and divergence-free so that
\begin{align}
\partial_t \Theta^j_{\lambda, \mu, \nu, \omega} & = - \lambda \omega \left( \left( \partial_j 	\varphi_\mu^j \right)_\lambda \circ \tau_{\omega t e_j} \right) \psi_\nu^j, \label{eq:time-derivative-theta} \\
\div W^j_{\lambda, \mu, \omega, \nu} & = \lambda \left(  (\partial_j \tilde{\varphi}_\mu^j )_\lambda \circ \tau_{\omega t e_j} \right)  \psi^j_\nu e_j, \label{eq:div-wj}
\end{align}
	holds and, because of the fact that $ Q^j=\frac{1}{\omega} \Theta^jW^j $, we still have a set of functions similar to a solution to the transport equation, as stated in the following proposition.
\end{rem}

Set 
\begin{equation}
	\epsilon \coloneqq \frac{d}{p}+\frac{d}{\tilde{p}}-d-1 = \frac{d}{\tilde{p}}-\frac{d}{p'}-1 >0.
	\label{eq:def-epsilon}	
\end{equation}
Note that $\epsilon>0$, because of \eqref{eq:p-cond}. 
\begin{prop}
	Define the global constants $ M $ (not depending on $p, \tilde p$) by
	\begin{align}
	M &\coloneqq 2d \max_{k=0,1} \left\lbrace  \|D^k \varphi\|_{L^\infty} \|D^k \psi\|_{L^\infty} , \ \|\varphi\|^2_{L^\infty} \|\psi\|^2_{L^\infty} \right\rbrace.
	\label{eq:def-M} 
	\end{align}
	The Mikado functions obey the following bounds:
	\begin{subequations}
	\begin{align}
		\left\|\Theta_{\lambda,\mu,\omega,\nu}^j(t) \right\|_{L^p} & \le \frac{M}{2d}, &  
		\left\| W_{\lambda,\mu,\omega,\nu}^j(t) \right\|_{L^{p'}} &\le \frac{M}{2d}, 	&	
		\left\|Q_{\lambda,\mu,\omega,\nu}^j(t) \right\|_{L^p} & \le  \frac{M \mu^b}{\omega},
		\label{eq:Mikado-bound-Lp} \\[1em]
		\left\|\Theta_{\lambda,\mu,\omega,\nu}^j(t) \right\|_{L^1} &\le \frac{M}{\mu^b}, &
		\left\|W_{\lambda,\mu,\omega,\nu}^j(t) \right\|_{L^1} &\le \frac{M}{\mu^a}, &	
		\left\|Q_{\lambda,\mu,\omega,\nu}^j(t) \right\|_{L^1} &\le \frac{M}{\omega}, 
		\label{eq:Mikado-bound-L1}
	\end{align}
	\begin{align}
		\left\|W_{\lambda,\mu,\omega,\nu}^j(t) \right\|_{C^0} &\le M \mu^b, 
		\label{eq:Mikado-bound-Linfty} \\[1em]
		\left\| W^j_{\lambda,\mu,\omega,\nu}(t) \right\|_{W^{1,\tilde{p}}} &\le M \frac{\lambda\mu+\nu}{\mu^{1+\epsilon}}. 
		\label{eq:Mikado-bound-WinW}
	\end{align}
	\end{subequations}
	Furthermore, for every $i \neq j$,
	\begin{equation}
	\label{eq:Mikado-disjoint-support}
	\Theta^i_{\lambda, \mu, \omega, \nu} W^j_{\lambda, \mu, \omega, \nu} = 0
	\end{equation}		
	and  the Mikado functions `solve the continuity equation' in the sense that
	\begin{equation}
	\partial_t Q^j_{\lambda,\mu,\omega,\nu} + \div \left(\Theta^j_{\lambda,\mu,\omega,\nu}W^j_{\lambda,\mu,\omega,\nu}\right) = 0
	\label{eq:Mikado-cancellation}
	\end{equation}
	on $ [0,T] \times\Td $. 
	\label{prop:mikados}
\end{prop}

\begin{proof}
	The inequalities in \eqref{eq:Mikado-bound-Lp}-\eqref{eq:Mikado-bound-L1}-\eqref{eq:Mikado-bound-Linfty} are immediate consequence  of \eqref{eq:phi-j-mu-scaling}. We show only the first inequality in \eqref{eq:Mikado-bound-Lp}, the other ones being completely similar:
	\begin{equation*}
	\begin{aligned}
	\left\| \Theta^j_{\lambda, \mu, \omega, \nu}(t) \right\|_{L^p}  
	& \leq \left\|  (\varphi_\mu^j)_\lambda \circ \tau_{\omega t e_j} \right\|_{L^p} \|\psi_\nu^j\|_{L^\infty}  \\
	& = \left\|  \varphi_\mu^j \right\|_{L^p} \|\psi^j\|_{L^\infty}  \\
	& = \left\|  \varphi \right\|_{L^p} \|\psi\|_{L^\infty}  \\
	& \leq \left\|  \varphi \right\|_{L^\infty} \|\psi\|_{L^\infty}  \\
	& \leq \frac{M}{2d}.
	\end{aligned}
	\end{equation*}
%
	Inequality \eqref{eq:Mikado-bound-WinW} requires direct calculation: using \eqref{eq:osc}, we get
	\begin{align*}
		\Big\|  W^j_{\lambda,\mu,\omega,\nu}(t) \Big\|_{W^{1,\tilde{p}}} 
		& \le \left\| (\tilde \varphi_\mu^j)_\lambda \circ \tau_{\omega t e_j} \right\|_{L^{\tilde p}} \left\| \psi^j_\nu \right\|_{L^\infty} 
		+ \left\| D \left((\tilde \varphi_\mu^j)_\lambda \circ \tau_{\omega t e_j} \right) \right\|_{L^{\tilde p}} \left\| \psi^j_\nu \right\|_{L^\infty}   \\
		& \qquad + \left\| (\tilde \varphi_\mu^j)_\lambda \circ \tau_{\omega t e_j} \right\|_{L^{\tilde p}} \left\|D (\psi^j_\nu) \right\|_{L^\infty}
		\\
		&\le \|\tilde{\varphi}_\mu^j\|_{L^{\tilde{p}}} \|\psi^j\|_{L^\infty} + \lambda \left\|D \tilde{\varphi}_\mu^j \right\|_{L^{\tilde{p}}} \|\psi^j\|_{L^\infty} + \nu \|\tilde{\varphi}_\mu^j\|_{L^{\tilde{p}}} \|D\psi^j\|_{L^\infty} \\
\text{(by \eqref{eq:phi-j-mu-scaling})}
		&\le \mu^{d/p' - d/\tilde p} \|\varphi\|_{L^{\tilde p}} \|\psi\|_{L^\infty} + \lambda \mu^{d/p' - d/\tilde p + 1} \|D \varphi\|_{L^{\tilde p}} \|\psi\|_{L^\infty} \\ 
		& \qquad + \nu \mu^{d/p' - d/\tilde p} \|\varphi\|_{L^{\tilde p}} \|\psi\|_{L^\infty} \\
		&\le M \left(\frac{\nu}{\mu^{1+\epsilon}} + \frac{\lambda}{\mu^\epsilon}\right).
	\end{align*}
	Equality \eqref{eq:Mikado-disjoint-support} is an immediate consequence of \eqref{eq:disjoint-support}. To prove \eqref{eq:Mikado-cancellation}, we observe that
	\begin{equation*}
	\Theta^j_{\lambda, \mu, \omega, \nu}(t,x) W^j_{\lambda, \mu, \omega, \nu}(t,x) = \omega Q^j_{\lambda, \mu, \omega, \nu}(t,x) e_j = F(x - \omega t e_j) \psi_\nu^j(x) e_j, 
	\end{equation*}
	for some $F: \Td \to \R$, whose precise form is not important. Since $\psi_\nu^j e_j$ is time independent and divergence free, we get
	\begin{equation*}
	\begin{aligned}
	\div \left( \Theta^j_{\lambda, \mu, \omega, \nu} W^j_{\lambda, \mu, \omega, \nu} \right) & = \nabla F \cdot \psi_\nu^j e_j,  \\
	\partial_t Q^j_{\lambda, \mu, \omega, \nu} & = -  \nabla F  \cdot \psi_\nu^j e_j,
	\end{aligned}	
	\end{equation*}
	and thus \eqref{eq:Mikado-cancellation} holds. 
%
%
%
\end{proof}

\subsection{Definition of perturbations}
\label{subsec:construct}
Given $(\rho_0,u_0,R_0)$ as in Proposition~\ref{prop:main}, we denote by $R_0^j(t,x)$ the components of the vector $R_0(t,x)$, i.e.
\begin{equation*}
R_0(t,x) = \sum_{j=1}^d R_0^j(t,x) e_j.
\end{equation*}
We now define the new density and velocity field as
\begin{align*}
\rho_1(t,x) &\coloneqq \rho_0(t,x) + \vartheta(t,x) + \vartheta_c(t) + q(t,x) + q_c(t)\\
u_1(t,x) &\coloneqq u_0(t,x) + w(t,x) + w_c(t,x)
\end{align*}
where $\vartheta$, $q$ and $w$ are the Mikado density, quadratic corrector term and Mikado flow weighted by the defect field $R_0$, defined as follows:
\begin{align*}
\vartheta(t,x) &\coloneqq \eta\sum_{j=1}^d \chi_j(t,x) \operatorname{sgn}\left(R_0^j(t,x)\right) \left|R_0^j(t,x)\right|^{1/p} \Theta_{\lambda,\mu,\omega,\nu}^j(t,x), \\
w(t,x) &\coloneqq \frac{1}{\eta}\sum_{j=1}^d \chi_j(t,x) \left|R_0^j(t,x)\right|^{1/p'} W_{\lambda,\mu,\omega,\nu}^j(t,x), \\
q(t,x) &\coloneqq \sum_{j=1}^d \chi_j^2(t,x) R_0^j(t,x) Q_{\lambda,\mu,\omega, \nu}^j(t,x).
\end{align*}
Here $\lambda, \mu, \omega, \nu$ will be chosen in Section \ref{sec:proof} to conclude the proof of Proposition \ref{prop:main}, the $\chi_j:[0,T]\times\Td\to[0,1]$ are cut-off functions which ensure the smoothness of the perturbations at the zero set of $R_0^j$:
\[ \chi_j(t,x) =
\begin{cases} 0 \text{ if } |R_0^j(t,x)| \le \frac{\delta}{4d}, \\
1 \text{ if } |R_0^j(t,x)| \ge \frac{\delta}{2d},
\end{cases}
\]
and $ \eta $ and $ \delta $ are the strictly positive numbers which appear in the statement of Proposition~\ref{prop:main}. 

The parameters $ \lambda,\mu,\omega,\nu \gg1$ will be fixed in Section \ref{sec:proof}. We will however use the shorter notation
\begin{equation*}
\begin{aligned}
\vartheta(t) & \coloneqq \sum_j a_j(t) \Theta^j(t), &
w(t) & \coloneqq \sum_j b_j(t) W^j(t), &
q(t) & \coloneqq \sum_j a_j(t) b_j(t) Q^j(t),
\end{aligned}
\end{equation*}
where
\begin{equation*}
\begin{aligned}
a_j(t) & := \eta \chi_j(t) \operatorname{sgn}\left(R_0^j(t)\right) \left|R_0^j(t)\right|^{1/p}, &
b_j(t) & := \frac{1}{\eta}  \chi_j(t)  \left|R_0^j(t)\right|^{1/p'}.
\end{aligned}
\end{equation*}
Notice that
\begin{equation*}
a_j(t) b_j(t) = \chi_j^2(t) R^j(t),
\end{equation*}
and the following estimates holds true:
\begin{subequations}
\label{eq:aj-bj}
\begin{equation}
\label{eq:aj-bj-lp}
\|a_j(t)\|_{L^p} \leq \eta \|R_0(t)\|_{L^1}^{1/p}, \quad \|b_j(t)\|_{L^{p'}} \leq \eta^{-1} \|R_0(t)\|_{L^1}^{1/p'}
\end{equation}
and, for every $k \in \N$, 
\begin{equation}
\label{eq:aj-bj:derivatives}
\|a_j\|_{C^k}, \|b_j\|_{C^k} \leq C(\eta, \delta, \|R_0\|_{C^k}). 		
\end{equation}
\end{subequations}
The corrector terms $\vartheta_c$, $ q_c $ are needed for $\rho_1$ to have zero mean value:
\begin{align*}
\vartheta_c(t) &\coloneqq -\fint_{\Td} \vartheta(t,x)dx \\
q_c(t) &\coloneqq -\fint_{\Td} q(t,x)dx.
\end{align*}
The corrector term $w_c$ is needed for $u_1$ to be divergence-free. We first compute
\begin{align*}
	\div w(t) &= \sum_j \div ( a_j(t) W^j(t)) \\
	& = \sum_j \div \left( a_j(t) \left( \left(\tilde \varphi_\mu^j\right)_\lambda \circ \tau_{\omega t e_j} \right) \psi^j_\nu e_j \right) \\
	& = \sum_j \nabla \Big( a_j(t) \, \left(\tilde \varphi_\mu^j\right)_\lambda \circ \tau_{\omega t e_j} \Big) \cdot e_j   \, \psi_\nu^j.
\end{align*}
We thus define
\begin{align}
\label{eq:wc-definition}
w_c(t) &\coloneqq - \sum_j \adiv_N \left( f_j(t),  \,  \psi^j_\nu\right), 
\end{align}
where we set for simplicity
	\begin{equation}
	\label{eq:fj}
	f_j(t)\coloneqq  \nabla \Big( a_j(t) \, \left(\tilde \varphi_\mu^j\right)_\lambda \circ \tau_{\omega t e_j} \Big)
	\end{equation}		
and $N$ is some large integer, which will be chosen in Section \ref{sec:proof} together with the parameters $\lambda, \mu, \omega, \nu$. Notice that this definition of the corrector $ w_c $ really cancels the divergence of $ w $.

\subsection{Estimates on the perturbations}
In this section we will formulate and prove all the necessary estimates on the perturbations, beginning with the density terms.
\begin{rem}
	In this and in the next two sections, Sections \ref{sec:error} and \ref{sec:proof}, we will denote by $C$ any constant which can depend on the constant $M$ defined in \eqref{eq:def-M}, on all the parameters in the statement of Proposition \ref{prop:main}, i.e.
	\begin{equation*}
	p, \tilde p, \delta, \eta, \rho_0, u_0, R_0,
	\end{equation*}
	on the parameter $N$ to be fixed in Section \ref{sec:proof} (and on the properties of the functions $\phi, \psi$ fixed in Section \ref{subsec:mikado},  in particular their derivatives and antiderivatives up to order $ N $ as in the definition of $ w_c $),	but not on 
	\begin{equation*}
	\lambda, \mu, \omega, \nu.
	\end{equation*}		
\end{rem}

	\begin{lem}[$ \vartheta $ in $L^p$ comparable to $ R_0 $]
	It holds
	\begin{equation}
	\|\vartheta(t)\|_{L^p} \le \frac{M\eta}{2} \|R_0(t)\|_{L^1}^{1/p} + \frac{C}{\lambda^{1/p}}.
	\label{eq:theta-in-Lp}
	\end{equation}
	\label{lem:thetaLp}
\end{lem}
\begin{proof}
	Applying the improved Hölder inequality \eqref{eq:improvedH} with $ f= a_j(t)$ and $ g_\lambda=\Theta^j(t)$ (recall that $\Theta^j(t)$ is $1/\lambda$-periodic, as $\nu$ is an integer multiple of $\lambda$) we obtain
	\begin{align*}
		\|\vartheta(t)\|_{L^p} 
		& \le \|a_j(t)\|_{L^p} \left\|\Theta^j(t)\right\|_{L^p} + \frac{C_p}{\lambda^{1/p}} \left\|a_j \right\|_{C^1} \left\|\Theta^j(t)\right\|_{L^p} \\
		\text{(by \eqref{eq:Mikado-bound-Lp} and \eqref{eq:aj-bj})} 
		& \le \frac{M\eta}{2d} \left\|R_0(t)\right\|_{L^1}^{1/p} + \frac{C}{\lambda^{1/p}}. 
	\end{align*}
	Summing over $j$, we get the desired inequality.
\end{proof}

\begin{lem}[$ q $ small in $L^p$]
It holds
	\begin{equation}
	\|q(t)\|_{L^p} \le  \frac{C\mu^b}{\omega}.
	\label{eq:q-in-Lp}
	\end{equation}
	\label{lem:qLp}
\end{lem}
\begin{proof}
	We obtain \eqref{eq:q-in-Lp} simply from the Hölder inequality, using \eqref{eq:Mikado-bound-Lp} and \eqref{eq:aj-bj:derivatives}:
	\begin{equation*}
	\|q(t)\|_{L^p} \le \sum_j \|a_j b_j\|_{C^0} \|Q^j(t)\|_{L^p} \le C \frac{\mu^b}{\omega} \raggedright\qedhere
	\end{equation*}
\end{proof}

\begin{lem}[$ \vartheta_c $ and $ q_c $ small as numbers]
It holds
	\begin{align}
		\left| \vartheta_c(t) \right| &\le C \mu^{-b},
		\label{eq:thetaC}\\
		\left| q_c(t) \right| &\le C \omega^{-1}.
		\label{eq:qC}
	\end{align}
	\label{lem:thetaC-qC}
\end{lem}

\begin{proof}
	Clearly the correctors are bounded by the $ L^1 $-norm of $ \vartheta(t)$ and $ q(t) $, so \eqref{eq:thetaC} and \eqref{eq:qC} follow immediately from \eqref{eq:Mikado-bound-L1} and \eqref{eq:aj-bj:derivatives}:
	\[ \left| \vartheta_c(t) \right| \leq \|\vartheta(t)\|_{L^1} \leq C \mu^{-b}, \qquad  \left|q_c(t)\right| \le \|q(t)\|_{L^1} \le C \omega^{-1}. \qedhere \]
\end{proof}

\begin{lem}[$ w $ in $L^{p'}$ comparable to $ R_0 $]
It holds
	\begin{equation}
	\|w(t)\|_{L^{p'}} \le \frac{M}{2\eta} \|R_0(t)\|_{L^1}^{1/p'} + \frac{C}{\lambda^{1/p'}}.
	\label{eq:w-in-Lp'}
	\end{equation}
	\label{lem:wLp'}
\end{lem}
\begin{proof}
	The proof is completely analog to the one of \eqref{eq:theta-in-Lp} and is thus omitted.
\end{proof}

\begin{lem}[$ w $ small in $ W^{1,\tilde{p}} $]
It holds
	\begin{equation}
	\|w(t)\|_{W^{1,\tilde{p}}} \le  \frac{C(\lambda \mu+\nu)}{\mu^{1+\epsilon}}.
	\label{eq:w-in-W1p}
	\end{equation}
	\label{lem:wW1p}
\end{lem}
\begin{proof}
	We only use Hölder together with \eqref{eq:Mikado-bound-WinW} and \eqref{eq:aj-bj:derivatives} and obtain
	\begin{align*}
		\|w(t)\|_{W^{1,\tilde{p}}} &\le \sum_j \left\| b_j(t) W^j(t) \right\|_{W^{1,\tilde{p}}} \\
		&\le \sum_j \left\| b_j \right\|_{C^1} \left\| W^j(t)\right\|_{W^{1,\tilde{p}}} \\
		&\le  \frac{C(\lambda \mu+\nu)}{\mu^{1+\epsilon}}. \qedhere
	\end{align*}
\end{proof}

\begin{lem}[Estimates on $f_j$]
\label{lem:fj}
For every $k, h \in \N$ and $r \in [1, \infty]$
\begin{equation*}
\left\| \D^k D^h f_j(t) \right\|_{L^r}  \leq  C(\lambda\mu)^{k+h+1} \mu^{b - d/r}. 
\end{equation*}
\end{lem}
\begin{proof}
Recalling the definition of $f_j$ in \eqref{eq:fj}, we have
	\begin{align*}
		\|\D^k D^h f_j(t)\|_{L^{r}} 
		& \leq \|f_j(t)\|_{W^{k+h,r}} \\
		& \leq \left\|a_j(t) \left( (\tilde \varphi_\mu^j)_\lambda \circ \tau_{\omega t e_j} \right) \right\|_{W^{k+h+1, r}} \\
		& \leq C \|a_j\|_{C^{k+h+1}} \|(\tilde \varphi_\mu^j)_\lambda\|_{W^{k+h+1,r}} \\
		\text{(by \eqref{eq:aj-bj:derivatives})}
		& \leq C \lambda^{k+h+1} \|\tilde \varphi_\mu^j\|_{W^{k+h+1, r}} \\
		\text{(by \eqref{eq:phi-j-mu-scaling})}		
		&\le C (\lambda\mu)^{k+h+1} \mu^{b - d/r}.  \qedhere
	\end{align*}
\end{proof}

\begin{lem}[$w_c$ small in $L^{p'}$]
It holds
	\begin{equation}
	\|w_c(t)\|_{L^{p'}} \le C \left( \sum_{k=1}^N \left(\frac{\lambda \mu}{\nu}\right)^k +\frac{(\lambda \mu)^{N+1}}{\nu^N} \right).
	\label{eq:wC-in-Lp'}
	\end{equation}
	\label{lem:wcLp'}
\end{lem}
\begin{proof}
Applying \eqref{eq:antidiv-in-p} to the definition \eqref{eq:wc-definition} of $ w_c $ we immediately obtain
	\begin{align*}
		\|w_c(t)\|_{L^{p'}} &\le \sum_j C \|\psi\|_{L^\infty} \left(\sum_{k=0}^{N-1} \frac{\|\D^kf_j(t)\|_{L^{p'}}}{\nu^{k+1}} + \frac{\|\D^Nf_j(t)\|_{L^{p'}}}{\nu^N}\right).
	\end{align*}
The conclusion follows applying Lemma \ref{lem:fj} with $h=0$, $r = p'$ and recalling that $b = d/p'$.
\end{proof}

\begin{lem}[$w_c$ small in $W^{1,\tilde{p}}$]
It holds
	\begin{equation}
	\|w_c(t)\|_{W^{1,\tilde{p}}} \le C \frac{\lambda\mu+\nu}{\mu^{1+\epsilon}} \left(\sum_{k=1}^N \left(\frac{\lambda\mu}{\nu}\right)^k + \frac{(\lambda\mu)^{N+1}}{\nu^N}\right).
	\label{eq:wC-in-W1p}
	\end{equation}
	\label{lem:wcW1p}
\end{lem}

\begin{proof}
	We will only estimate $ \|Dw_c(t)\|_{L^{\tilde{p}}} $ as the estimate on $\|w_c(t)\|_{L^{\tilde{p}}}$ is very similar to the proof of the previous lemma (we just gain a factor of $ \mu^{-(1+\epsilon)} $ because of the integrability of $ \tilde{\varphi}_\mu $). 	By statement~(ii) of Lemma~\ref{lem:adiv-properties} we can split $ Dw_c $ into:
	\begin{align*}
		Dw_c(t) &= -\sum_j D\adiv_N \left( f_j(t), \psi_\nu^j \right)  = - \sum_j \adiv_N \left( D f_j(t), \psi_\nu^j \right) + \adiv_N \left( f_j, D \left( \psi_\nu^j \right) \right).
	\end{align*}
Both terms can now be estimated analog to the previous lemma by application of \eqref{eq:antidiv-in-p}, resulting in (the constant may change from line to line)
	\begin{align*}
		\|Dw_c(t)\|_{L^{\tilde{p}}} &\le C\sum_j \Bigg[ \|\psi\|_{L^\infty} \left(\sum_{k=0}^{N-1} \frac{\|\D^k Df_j(t)\|_{L^{\tilde p}}}{\nu^{k+1}} + \frac{\|\D^N Df_j(t)\|_{L^{\tilde p}}}{\nu^N}\right) \\
		&\hspace{6em} + \|D\psi_\nu\|_{L^\infty} \left(\sum_{k=0}^{N-1} \frac{\|\D^kf_j(t)\|_{L^{\tilde p}}}{\nu^{k+1}} + \frac{\|\D^Nf_j(t)\|_{L^{\tilde p}}}{\nu^N}\right)\Bigg] \\
		\text{(by Lemma \ref{lem:fj})}\quad
		&\le C \Bigg[ \mu^{d/p'-d/\tilde{p}} \left(\sum_{k=1}^N \frac{(\lambda\mu)^{k+1}}{\nu^k} + \frac{(\lambda\mu)^{N+2}}{\nu^N}\right) \\
		&\hspace{10em} + \nu \mu^{d/p'-d/\tilde{p}} \left( \sum_{k=1}^N \frac{(\lambda\mu)^k}{\nu^k} + \frac{(\lambda\mu)^{N+1}}{\nu^N} \right) \Bigg] \\
		&= C\frac{\lambda\mu+\nu}{\mu^{1+\epsilon}} \left(\sum_{k=1}^N \left(\frac{\lambda\mu}{\nu}\right)^k + \frac{(\lambda\mu)^{N+1}}{\nu^N}\right). \qedhere
	\end{align*}
\end{proof}

\section{The new defect field}
\label{sec:error}

\subsection{Definition of  $R_1$}

Given the perturbations defined in the previous section we now have to find a vector field $ R_1 $ so that $ (\rho_1,u_1,R_1) $ solve \eqref{eq:cont-defect} on $ [0,T]\times\Td $. This is achieved basically by taking the antidivergence of the left hand side of \eqref{eq:cont-defect}, but as we want to show that $ R_1 $ can be chosen arbitrarily small in $ L^1 $ in order to prove \eqref{eq:R1}, we need to be careful about the exact form of the antidivergence. Therefore, decompose the left hand side of \eqref{eq:cont-defect} as
\begin{align}
-\div R_1 &= \partial_t \rho_1 + \div(\rho_1 u_1) \nonumber \\
&=\underbrace{\partial_t \rho_0 + \div(\rho_0 u_0)}_{\mathrlap{=-\div R_0}} + \partial_t(\vartheta+\vartheta_c+q+q_c) + \div(\rho_0 (w+w_c)) + \div((\vartheta+q) u_0) \nonumber\\
 &\qquad + \div((\vartheta+q)(w+w_c)) + \underbrace{(\vartheta_c+q_c)\div((u_0+w+w_c))}_{\mathrlap{= 0 \text{ by def.~of } w_c}} \nonumber\\
 \begin{split}
 &= \partial_t (q +q_c)  + \div(\vartheta w-R_0) \\
 &\qquad + \partial_t (\vartheta + \vartheta_c) + \div(\rho_0 w + \vartheta u_0) \\
 &\qquad + \div (q(u_0+w)) \\
 &\qquad + \div ((\rho_0+\vartheta+q)w_c).
 \label{eq:divR1}
 \end{split}
\end{align}

In the next sections we analyze each line in \eqref{eq:divR1} separately. In particular we will define and estimate $R^{\chi}$ (in \eqref{eq:R-chi}), $R^{\mathrm{time}, 1}$ (in \eqref{eq:R-time1}), $R^{\mathrm{quadr}}$ (in \eqref{eq:R:quadr}), $R^{\mathrm{lin}}$ (in \eqref{eq:R-lin}), $R^{\mathrm{time}, 2}$ (in \eqref{eq:R-time2}), $R^q$ (in \eqref{eq:Rq}), $R^{\mathrm{corr}}$ (in \eqref{eq:Rcorr}), so that
\begin{equation*}
\begin{aligned}
\partial_t (q +q_c)  + \div(\vartheta w-R_0) & = \div R^{\mathrm{time}, 1} + \div R^{\mathrm{quadr}} + \div R^{\chi},   \\
\partial_t (\vartheta + \vartheta_c) + \div(\rho_0 w + \vartheta u_0) & =  \div R^{\mathrm{time}, 2}+ \div R^{\mathrm{lin}},  \\
\div (q(u_0+w)) & = \div R^q, \\
\div ((\rho_0+\vartheta+q)w_c) & = \div R^{\mathrm{corr}},
\end{aligned}
\end{equation*}
and thus
\begin{equation*}
- \div R_1 = \partial_t \rho_1 + \div (\rho_1 u_1)
\end{equation*}
for
\begin{equation*}
- R_1 :=  R^{\mathrm{time}, 1} +  R^{\mathrm{quadr}} + R^{\chi}  + R^{\mathrm{time}, 2}+  R^{\mathrm{lin}} + R^q + R^{\mathrm{corr}}. 
\end{equation*}

\subsection{Analysis of the first line in \eqref{eq:divR1}}

We write
\begin{equation*}
R_0 = \sum_j R_0^j e_j = \sum_j (1 - \chi_j^2) R_0^j e_j + \sum_j \chi_j^2 R_0^j e_j
\end{equation*}
and thus
\begin{equation*}
\begin{split}
- \div R_0 	& = \div \bigg(  R^{\chi} -  \sum_j \chi_j^2 R_0^j e_j \bigg) \\
					& = \div R^\chi - \sum_j \nabla (\chi_j^2 R_0^j) \cdot e_j \\
					& =  \div R^{\chi} -  \sum_j \nabla ( a_j b_j ) \cdot e_j
\end{split}
\end{equation*}
where we set
\begin{equation}
\label{eq:R-chi}
R^\chi : = - \sum_j (1 - \chi_j^2) R_0^j e_j.
\end{equation}
Observe now that, because of \eqref{eq:Mikado-disjoint-support}, 
\begin{equation*}
\vartheta w = \sum_j a_j b_j \Theta^j W^j = \sum_j \chi_j^2 R_0^j \Theta^j W^j.
\end{equation*}
Therefore 
\begin{equation*}
\begin{aligned}
\div ( \vartheta w )&  = \sum_j a_j b_j \div(\Theta^j W^j) + \nabla (a_j b_j) \cdot  \Theta^j W^j \\
\end{aligned}
\end{equation*}
and thus
\begin{equation}
\label{eq:div:thetaw:minus:R}
\begin{aligned}
 \div  (\vartheta w - R_0) 
& =  \sum_j a_j b_j \div(\Theta^j W^j) + \nabla (a_j b_j) \cdot  \Theta^j W^j + \div R^\chi -  \sum_j \nabla( a_j b_j ) \cdot e_j \\
& =   \sum_j a_j b_j \div(\Theta^j W^j) + \nabla (a_j b_j) \cdot  \big[  \Theta^j W^j  - e_j \big] + \div R^\chi \\
& =  \sum_j a_j b_j \div(\Theta^j W^j) \\
& \qquad + \bigg(  \nabla (a_j b_j) \cdot  \big[  \Theta^j W^j  - e_j \big]  - \fint  \nabla (a_j b_j) \cdot  \big[  \Theta^j W^j  - e_j \big]  \bigg) \\
& \qquad +  \fint  \nabla (a_j b_j) \cdot  \big[  \Theta^j W^j  - e_j \big] \\ 
& \qquad + \div R^\chi.
\end{aligned}
\end{equation}
On the other side
\begin{equation}
\label{eq:partial:q}
\begin{aligned}
\partial_t & (q +q_c)  \\
& = \sum_j a_j b_j \partial_t Q^j +  \partial_t (a_j b_j) \, Q^j + q_c' \\
& = \sum_j a_j b_j \partial_t Q^j  + \bigg(  \partial_t (a_j b_j) \, Q^j - \fint  \partial_t (a_j b_j) \, Q^j \bigg) + \bigg( \fint  \partial_t (a_j b_j) \, Q^j + q_c' \bigg).
\end{aligned}
\end{equation}
Putting together \eqref{eq:div:thetaw:minus:R} and \eqref{eq:partial:q} we get
\begin{equation*}
\begin{aligned}
\partial_t (q +q_c)   + \div (\vartheta w - R_0)  & = \sum_j a_j b_j \underbrace{\big[  \partial_t Q^j + \div (\Theta^j W^j) \big]}_{= 0 \text{ by \eqref{eq:Mikado-cancellation}}} \\
& \qquad + \bigg(  \partial_t (a_j b_j) \, Q^j - \fint  \partial_t (a_j b_j) \, Q^j \bigg) \\
& \qquad + \bigg(  \nabla (a_j b_j) \cdot  \big[  \Theta^j W^j  - e_j \big]  - \fint  \nabla (a_j b_j) \cdot  \big[  \Theta^j W^j  - e_j \big]  \bigg) \\
& \qquad + \div R^\chi  \\
& \qquad + \underbrace{ \fint  \partial_t (a_j b_j) \, Q^j + q_c' +  \fint  \nabla (a_j b_j) \cdot  \big[  \Theta^j W^j  - e_j \big]}_{\substack{= 0, \text{ as the l.h.s. has zero mean value} \\ \text{and each other line in the r.h.s. has zero mean value}}} \\
& = \sum_j \bigg(  \partial_t (a_j b_j) \, Q^j - \fint  \partial_t (a_j b_j) \, Q^j \bigg) \\
& \qquad + \bigg(  \nabla (a_j b_j) \cdot  \big[  \Theta^j W^j  - e_j \big]  - \fint  \nabla (a_j b_j) \cdot  \big[  \Theta^j W^j  - e_j \big]  \bigg) \\
& \qquad + \div R^\chi \\
& = \div R^{\mathrm{time}, 1} + \div R^{\mathrm{quadr}} + \div R^{\chi},
\end{aligned}
\end{equation*}
where $R^{\mathrm{time}, 1}$ is defined by
\begin{equation}
\label{eq:R-time1}
R^{\mathrm{time}, 1} := \sum_j \left\{ \mathcal{D}^{-1} \bigg( \partial_t (a_j b_j) \, Q^j - \fint_{\Td} \partial_t (a_j b_j)  \, Q^j \bigg) \right\},
\end{equation}
and $R^{\mathrm{quadr}}$ is defined  in such a way that
\begin{equation}
\label{eq:div:r:quadr}
\begin{aligned}
\div R^{\mathrm{quadr}} & =  \sum_j \left\{ \nabla (a_j b_j) \cdot \big[  \Theta^j W^j - e_j \big] -  \fint \nabla R_0^j \cdot \big[  \Theta^j W^j - e_j \big] \right\}.
\end{aligned}
\end{equation}
as follows. 
We first compute
\begin{equation*}
\begin{aligned}
\nabla (a_j b_j) \cdot \big[  \Theta^j W^j - e_j \big] 
& = \nabla(a_j b_j) \cdot \left[  \left(\varphi_\mu^j \tilde{\varphi}_\mu^j\right)_\lambda \circ \tau_{\omega t e_j}  (\psi_\nu^j)^2 - 1 \right] e_j \\
& = \nabla(a_j b_j) \cdot  \Big[ (\varphi_\mu^j \tilde{\varphi}_\mu^j)_\lambda \circ \tau_{\omega t e_j}  \left((\psi_\nu^j)^2-1\right) + \big( (\varphi_\mu^j \tilde{\varphi}_\mu^j)_\lambda \circ \tau_{\omega t e_j}  -1 \big) \Big] e_j\\
& = \nabla(a_j b_j) \cdot  \Big[ (\varphi_\mu^j \tilde{\varphi}_\mu^j)_\lambda \circ \tau_{\omega t e_j}  \left((\psi^j)^2-1\right)_\nu + \big( \varphi_\mu^j \tilde{\varphi}_\mu^j -1\big)_\lambda \circ \tau_{\omega t e_j}  \Big] e_j \\
& = \partial_j (a_j b_j)   \Big[ (\varphi_\mu^j \tilde{\varphi}_\mu^j)_\lambda \circ \tau_{\omega t e_j}  \left((\psi^j)^2-1\right)_\nu + \big( \varphi_\mu^j \tilde{\varphi}_\mu^j -1\big)_\lambda \circ \tau_{\omega t e_j}  \Big].
\end{aligned}
\end{equation*}
We then define 
\begin{equation*}
\begin{aligned}
R^{\mathrm{quadr}, 1} & \coloneqq \sum_j \adiv_1\Big(\partial_j (a_j b_j)  (\varphi_\mu^j \tilde{\varphi}_\mu^j)_\lambda \circ \tau_{\omega t e_j}, \ \big((\psi^j)^2-1\big)_\nu \Big), \\
R^{\mathrm{quadr}, 2} & \coloneqq \sum_j \adiv_1\left( \partial_j (a_j b_j) , \  \left(\varphi_\mu^j \tilde{\varphi}_\mu^j-1\right)_\lambda   \circ \tau_{\omega t e_j} \right),
\end{aligned}
\end{equation*}
and
\begin{equation}
\label{eq:R:quadr}
R^{\mathrm{quadr}} \coloneqq R^{\mathrm{quadr}, 1} + R^{\mathrm{quadr}, 2},
\end{equation}
so that \eqref{eq:div:r:quadr} holds. Notice that the definitions of $R^{\mathrm{quadr}, 1}$ and $R^{\mathrm{quadr},2}$ are well posed, as 
\begin{equation*}
\fint_{\Td} \big((\psi^j)^2-1\big)_\nu = 0, \qquad \fint_{\Td} \left(\varphi_\mu^j \tilde{\varphi}_\mu^j-1\right)_\lambda   \circ \tau_{\omega t e_j} = 0,
\end{equation*}
because of \eqref{eq:mean-value-1-j} and \eqref{eq:psi:mean:value}. We now estimate $R^{\chi}$, $R^{\mathrm{time}, 1}$, $R^{\mathrm{quadr}}$.

\begin{lem}[Bound on $ R^\chi $]
It holds
	\begin{equation}
	\left\|R^\chi(t)\right\|_{L^1} \le \frac{\delta}{2}.
	\end{equation}
	\label{lem:Rchi}
\end{lem}
\begin{proof}
	From the definition of $ \chi_j $ it is obvious that $ |R_0^j(t,x)|\le \frac{\delta}{2d} $ on the support of $(1 - \chi_j^2(t,x))$, so
	\[ \left\|R^\chi(t) \right\|_{L^1} \le \sum_j \int_{\operatorname{spt}(1 - \chi_j^2(t))} |R_0^j(t,x)| dx \le d \int_{\Td} \frac{\delta}{2d} \le \frac{\delta}{2}. \qedhere \]
\end{proof}

\begin{lem}[Bound on $R^{\mathrm{time}, 1}$]
It holds
	\begin{align}
	\left\|R^{\mathrm{time}, 1}(t)\right\|_{L^1} &\le C \frac{1}{\omega}.
	\end{align}
	\label{lem:Rqq13}
\end{lem}
\begin{proof}
Using the definition of $R^{\mathrm{time}, 1}$ in \eqref{eq:R-time1} and applying Lemma~\ref{lem:antiderivative-end}, we get
\begin{align*}
\left\|R^{\mathrm{time},1}(t)\right\|_{L^1} 
	& \le C \sum_j \|\partial_t(a_j(t) b_j(t)) Q^j(t)\|_{L^1} \\
	& \leq C \sum_j \|\partial_t (a_j b_j)\|_{C^0} \|Q^j(t)\|_{L^1} \\
\text{(by \eqref{eq:Mikado-bound-L1})}
	& \leq C \frac{1}{\omega}.  \qedhere
\end{align*}
\end{proof}

\begin{lem}[Bound on $ R^{\mathrm{quadr}} $]
It holds
	\begin{equation}
	\left \|R^{\mathrm{quadr}}(t)\right \|_{L^1} \le C \left( \frac{\lambda\mu}{\nu} +\frac{1}{\lambda} \right).
	\end{equation}
	\label{lem:Rquadr1}
\end{lem}

\begin{proof}
	First observe that both terms in the definition of $ R^{\mathrm{quadr}} $ need to be handled separately as the fast oscillation term of $R^{\mathrm{quadr}, 1}$ is $ (1/\nu) $-periodic whereas in $R^{\mathrm{quadr}, 2}$ there is only $ (1/\lambda) $-periodicity. For $R^{\mathrm{quadr}, 1}$,  \eqref{eq:antidiv-in-p} (with $N = 1$) and standard Hölder gives us
	\begin{align*}
		\|R^{\mathrm{quadr}, 1}(t) \|_{L^1} 
		&\le \frac{C}{\nu} \|\psi^2-1\|_{C^0} \Bigg( \left \|\partial_j \left( a_j(t) b_j(t) \right) \left( \varphi_\mu^j \tilde{\varphi}_\mu^j \right)_\lambda \circ \tau_{\omega t e_j} \right\|_{L^1} \\
		& \qquad \qquad \qquad \qquad  \qquad + \left\|\D^1\Big( \partial_j\big(  a_j(t) b_j(t) \big) \, (\varphi_\mu^j \tilde{\varphi}_\mu^j)_\lambda \circ \tau_{\omega t e_j} \Big) \right\|_{L^1} \Bigg)\\
		& \le \frac{C}{\nu} \Bigg( \left\| \partial_j \left(a_j(t) b_j(t) \right) \right\|_{C^0} \left\| \left(\varphi_\mu^j \tilde{\varphi}_\mu^j \right)_\lambda \circ \tau_{\omega t e_j} \right\|_{L^1} \\
		& \qquad \qquad 
		+  \left\| \partial_j \left(a_j(t) b_j(t) \right) \right\|_{C^1} \left\| \left(\varphi_\mu^j \tilde{\varphi}_\mu^j \right)_\lambda \circ \tau_{\omega t e_j} \right\|_{W^{1,1}} 
		\Bigg)
		\\
		&\le \frac{C}{\nu} \|a_j b_j \|_{C^2}  \left( \|\varphi_\mu^j \tilde{\varphi}_\mu^j \|_{L^1} +  \lambda\|\varphi_\mu^j \tilde{\varphi}_\mu^j \|_{W^{1,1}}\right) \\
		&\le  \frac{C \lambda\mu}{\nu},
	\end{align*}
	where in the last step we used \eqref{eq:phi-j-mu-scaling}. 	For $R^{\mathrm{quadr}, 2}$ we apply \eqref{eq:antidiv-in-infty} (again with $N=1$) and obtain
	\begin{align*}
		&\hspace{-2em} \|R^{\mathrm{quadr}, 2}(t)\|_{L^1} \\
		&\le C\|\varphi_\mu^j \tilde{\varphi}_\mu^j -1\|_{L^1} \left( \frac{1}{\lambda} \left\|\partial_j\left(\chi_j^2R_0^j\right)\right\|_{C_0} + \frac{1}{\lambda} \left\|\partial_j\left(\chi_j^2R_0^j\right)\right\|_{C^1} \right) \\
		&\le C \frac{1}{\lambda},
	\end{align*}
	as $\|\varphi_\mu^j \tilde{\varphi}_\mu^j \|_{L^1} = 1$, by \eqref{eq:phi-j-mu-scaling}. 	Together these two estimates supply the required bound.
\end{proof}

\subsection{Analysis of the second line in \eqref{eq:divR1}}
We have
\begin{equation*}
\begin{aligned}
\partial_t (\vartheta + \vartheta_c) + \div(\vartheta u_0 + \rho_0 w) 
& = \sum_j a_j \partial_t \Theta^j + (\partial_t a_j) \Theta^j + \div(\vartheta u_0 + \rho_0 w) + \vartheta_c' \\
& = \sum_j  \bigg( a_j \partial_t \Theta^j - \fint a_j \partial_t \Theta^j  \bigg) \\
& \qquad + \bigg( (\partial_t a_j) \Theta^j - \fint (\partial_t a_j) \Theta^j \bigg) + \div(\vartheta u_0 + \rho_0 w) \\
& \qquad  + \underbrace{\fint a_j \partial_t \Theta^j  + \fint (\partial_t a_j) \Theta^j + \vartheta_c'}_{\substack{=0 \text{ as the l.h.s. and each other line} \\ \text{in the r.h.s. has zero mean value} }} \\
& = \div R^{\mathrm{time}, 2} + \div R^{\mathrm{lin}},
\end{aligned}
\end{equation*}
where
\begin{equation}
\label{eq:R-lin}
R^{\mathrm{lin}} : = \D^{-1} \bigg(  (\partial_t a_j) \Theta^j - \fint (\partial_t a_j) \Theta^j \bigg) + \vartheta u_0 + \rho_0 w
\end{equation}
and $R^{\mathrm{time},2}$ is defined in such a way that
\begin{equation*}
\div R^{\mathrm{time}, 2} = \sum_j  \bigg( a_j \partial_t \Theta^j - \fint a_j \partial_t \Theta^j  \bigg),
\end{equation*}
as follows. 
Using \eqref{eq:time-derivative-theta}, we get
\begin{equation*}
a_j \partial_t \Theta^j = - \lambda \omega a_j \left( \left(\partial_j \varphi_\mu^j \right)_\lambda \circ \tau_{\omega t e_j} \right) \psi_\nu^j
\end{equation*}
and thus we can define
\begin{equation}
\label{eq:R-time2}
R^{\mathrm{time}, 2} \coloneqq - \lambda \omega \sum_j \adiv_N \Big( a_j \left(\partial_j \varphi_\mu^j \right)_\lambda \circ \tau_{\omega t e_j}, \ \psi_\nu^j \Big),
\end{equation}
where $N$ will be fixed in Section \ref{sec:proof}, as we have already stressed.

\begin{lem}[Bound on $ R^{\mathrm{lin}} $]
It holds
	\begin{equation}
	\left\|R^{\mathrm{lin}}(t)\right\|_{L^1} \le C \left( \frac{1}{\mu^a}+\frac{1}{\mu^b} \right).
	\end{equation}
	\label{lem:Rlin1}
\end{lem}
\begin{proof}
	For the first term in the definition \eqref{eq:R-lin} of $R^{\mathrm{lin}}$, Lemma~\ref{lem:antiderivative-end} yields
	\begin{align*}
		\left\|\D^{-1}\left( \partial_t a_j(t)  \Theta^j(t) - \fint \partial_t a_j(t) \Theta^j(t) \right)\right\|_{L^1} &\le C \|\partial_t a_j(t) \, \Theta^j(t)\|_{L^1} \\
		& \leq C \|\partial_t a_j\|_{C^0} \|\Theta^j(t)\|_{L^1} \\
\text{(by \eqref{eq:Mikado-bound-L1})}
		&\le  \frac{C}{\mu^b}.
	\end{align*}
	For the second term in the definition \eqref{eq:R-lin} of $R^{\mathrm{lin}}$, simply apply Hölder's inequality
	\begin{align*}
		\|\rho_0(t)w(t)\|_{L^1} &\le \|\rho_0\|_{C^0} \|a_j\|_{C^0} \|W^j(t)\|_{L^{1}} \\
\text{(by \eqref{eq:Mikado-bound-L1})}
		&\le \frac{C}{\mu^a}.
	\end{align*}
	The third term is handled completely analog, resulting in
	\[ \|\vartheta(t) u_0(t)\|_{L^1} \le  \frac{C}{\mu^b}. \]
	By adding the three terms we obtain the required bound.
\end{proof}

\begin{lem}[Bound on $ R^{\mathrm{time}, 2} $]
It holds
	\begin{equation}
	\left\|R^{\mathrm{time}, 2}(t)\right\| \le  C \frac{\omega}{\mu^b}\left(\sum_{k=1}^N \left(\frac{\lambda\mu}{\nu}\right)^k + \frac{(\lambda\mu)^{N+1}}{\nu^N}\right).
	\end{equation}
	\label{lem:Rlin2}
\end{lem}

\begin{proof}
	$ R^{\mathrm{time}, 2} $ is defined in \eqref{eq:R-time2} by application of the bilinear antidivergence operator $\adiv_N$ of \Cref{subsec:adiv} to the product of $a_j (\partial_j \varphi_\mu^j)_\lambda \circ \tau_{\omega t e_j}$  and $ \psi_\nu^j $, so \eqref{eq:antidiv-in-p} yields
	\begin{align*}
		\left\|R^{\mathrm{time}, 2}(t)\right\|_{L^1} 
		& \leq C \lambda \omega  \Bigg( \sum_{k=0}^{N-1} \frac{1}{\nu^{k+1}} \left\| \D^k \left( a_j(t) (\partial_j \varphi_\mu^j)_\lambda \circ \tau_{\omega t e_j} \right) \right\|_{L^1} + \frac{1}{\nu^{N}} \left\| \D^N \left( a_j(t) (\partial_j \varphi_\mu^j)_\lambda \circ \tau_{\omega t e_j} \right) \right\|_{L^1} \Bigg) \\
		& \leq C \lambda \omega  \Bigg( \sum_{k=0}^{N-1} \frac{1}{\nu^{k+1}} \left\| a_j(t) (\partial_j \varphi_\mu^j)_\lambda \circ \tau_{\omega t e_j} \right\|_{W^{k,1}} + \frac{1}{\nu^{N}} \left\| a_j(t) (\partial_j \varphi_\mu^j)_\lambda \circ \tau_{\omega t e_j} \right\|_{W^{N,1}} \Bigg) \\
		& \leq C \lambda \omega  \Bigg( \sum_{k=0}^{N-1} \frac{1}{\nu^{k+1}} \| a_j \|_{C^k} \left\|(\partial_j \varphi_\mu^j)_\lambda \right\|_{W^{k,1}} + \frac{1}{\nu^{N}} \| a_j \|_{C^N} \left\| (\partial_j \varphi_\mu^j)_\lambda \right\|_{W^{N,1}} \Bigg) \\
		& \leq C  \lambda \omega  \Bigg( \sum_{k=0}^{N-1} \frac{\left\|(\partial_j \varphi_\mu^j)_\lambda \right\|_{W^{k,1}} }{\nu^{k+1}} + \frac{\left\|(\partial_j \varphi_\mu^j)_\lambda \right\|_{W^{N,1}} }{\nu^{N}} \Bigg) \\
		\text{(by \eqref{eq:phi-j-mu-scaling})}
		& \leq C  \lambda \mu^{1 - b} \omega  \Bigg( \sum_{k=0}^{N-1} \frac{(\lambda \mu)^k}{\nu^{k+1}} + \frac{(\lambda \mu)^N}{\nu^{N}} \Bigg) \\
		&\le C \frac{\omega}{\mu^b} \left(\sum_{k=1}^N \left(\frac{\lambda\mu}{\nu}\right)^k + \frac{(\lambda\mu)^{N+1}}{\nu^N} \right),
	\end{align*}
	which is exactly the desired inequality.
\end{proof}

\subsection{Analysis of the third line in \eqref{eq:divR1}} 
We simply define 
\begin{equation}
\label{eq:Rq}
R^q \coloneqq q(u_0+w).
\end{equation}

\begin{lem}[Bound on $ R^q $]
It holds
	\begin{align}
	\left\|R^q(t)\right\|_{L^1} &\le C \frac{\mu^b}{\omega}.
	\end{align}
	\label{lem:Rqq1}
\end{lem}
\begin{proof}
	From the definitions of $ q $ and $ w $ we immediately get
	\begin{align*}
		\left\|R^q(t)\right\|_{L^1} &\le \|q(t)\|_{L^1}\left(\|u_0(t)\|_{C^0}+\|w(t)\|_{C^0}\right) \\
		& \leq \sum_j \|a_j b_j\|_{C^0} \|Q^j(t)\|_{L^1} \bigg( \|u_0\|_{C^0} + \sum_i \|b_i\|_{C^0} \|W^i(t)\|_{C^0} \bigg) \\
		& \leq C \sum_j \|Q^j(t)\|_{L^1} \bigg( 1 + \sum_i \|W^i(t)\|_{C^0} \bigg) \\
		\text{(by \eqref{eq:Mikado-bound-L1} and \eqref{eq:Mikado-bound-Linfty})} 
		& \leq  \frac{C}{\omega} (1 + \mu^b),
	\end{align*}
	which implies the desired inequality.
\end{proof}

\subsection{Analysis of the fourth line in \eqref{eq:divR1}}
We simply define
\begin{equation}
\label{eq:Rcorr}
R^{\mathrm{corr}}\coloneqq (\rho_0+\vartheta+q)w_c .
\end{equation}
%
%

\begin{lem}[Bound on $ R^{\mathrm{corr}} $]
It holds
	\begin{equation*}
	\left\|R^{\mathrm{corr}}(t)\right\|_{L^1} \le C \left(1+\frac{1}{\lambda^{1/p}}+\frac{\mu^b}{\omega}\right) \left( \sum_{k=1}^N \left(\frac{\lambda \mu}{\nu}\right)^k +\frac{(\lambda \mu)^{N+1}}{\nu^N} \right).
	\end{equation*}
	\label{lem:Rcorr}
\end{lem}
\begin{proof}
	The inequality is easier to prove than to state as it is an immediate consequence of Lemmata~\ref{lem:thetaLp},~\ref{lem:qLp} and~\ref{lem:wcLp'}. We omit the details.
\end{proof}

\section{Proof of the main proposition}
\label{sec:proof}

Given the estimates proven in Sections~\ref{sec:perturbation} and~\ref{sec:error} we are now able to prove Proposition~\ref{prop:main}. Let $ p\in(1,\infty) $ and $ \tilde{p}\in[1,\infty) $ so that \eqref{eq:p-cond} holds. Let $ \delta,\eta>0 $ and let
	\[ (\rho_0,u_0,R_0):[0,T]\times\Td\to\R\times\R^d\times\R^d \]
	be a smooth solution of the incompressible continuity-defect equation \eqref{eq:cont-defect}.

	\subsection{Choice of parameters}
	Recall that $ M $ was defined in \eqref{eq:def-M}. Let $ \epsilon $ be as in \eqref{eq:def-epsilon} and note that $ \epsilon>0 $ by \eqref{eq:p-cond}. Recall that $ a=d/p>0 $ and $ b=d/p'>0 $. For some large positive integer $ \lambda $ to be defined later:
	\begin{enumerate}
	\item Set $ \mu\coloneqq\lambda^\alpha $ for some $ \alpha(\epsilon)>2 \epsilon^{-1} > \epsilon^{-1} $.
	\item Set $ \nu\coloneqq\lambda^\gamma $ for a \emph{natural} number $ \gamma(\alpha,\epsilon) $ chosen such that
	\[ \alpha+1 < \gamma < \alpha (1+\epsilon) \]
	which is possible by the choice of $ \alpha $. In this way, $\nu$ is a multiple of $\lambda$ and the Mikado functions defined in Section \ref{subsec:mikado} are $\lambda$-periodic. 
	\item Choose $ \beta(b,\alpha,\gamma) $ such that
	\[ b\alpha<\beta<b\alpha+\gamma-(\alpha+1) \]
	which is possible by the first condition on $ \gamma $, and set $ \omega\coloneqq\lambda^\beta $.
	\item Finally, choose an integer $ N(\alpha,\gamma) $ which is large enough so that
	\[ \frac{N}{N-1} < \frac{\gamma}{1+\alpha} \]
	which is also possible by the first condition on $ \gamma $.
	\end{enumerate}
	Let us summarize the conditions imposed by our choice of the parameters $ \alpha,\beta,\gamma $ and $ N $:
	\begin{subequations}
		\begin{align}
		1&<\alpha\epsilon \label{eq:alpha-large}\\
		\alpha+1&<\gamma \label{eq:gamma-large}\\
		\gamma&<\alpha(1+\epsilon) \label{eq:gamma-small}\\
		b\alpha&<\beta \label{eq:beta-large}\\
		\beta+1+\alpha&<b\alpha+\gamma \label{eq:beta-small} \\
		N(1+\alpha) &< (N-1)\gamma. \label{eq:N-large}
		\end{align}
		\label{eq:parameters}
	\end{subequations}
	
	\subsection{Definition of the new solution} Let $ (\rho_1,u_1) $ be as defined in Section~\ref{sec:perturbation} and $ R_1 $ as in Section~\ref{sec:error}. Then $ (\rho_1,u_1,R_1) $ is a solution of \eqref{eq:cont-defect} as stated in the construction of $ R_1 $. Clearly the solution is smooth in time and space (ensured by the cut-offs $ \chi_j $) and it is equal to $ (\rho_0,u_0,R_0)(t) $ if $ R_0(t)\equiv0 $ holds, as the construction is completely local in time apart from the definition of $ R^{\mathrm{lin}} $ and $ R^{\mathrm{time},1} $, which contain the time derivative of $ R_0 $. However, by the definition of the cut-off functions $ \chi_j $ it is clear that
	\[ R_0(t)\equiv0 \implies \partial_t a_j = \partial_t \left(\chi_j(t,\cdot)|R_0^j(t,\cdot)|^{1/p}\right) \equiv 0 \]
	and and analog for $ \partial_t (a_j b_j) $, so also $ R^{\mathrm{lin}}(t), R^{\mathrm{time},1}(t)\equiv 0 $ holds.
	
	We need to show \eqref{eq:rho1-rho0_Lp}--\eqref{eq:R1}, which is equivalent to
	\begin{subequations}
		\begin{align}
		\|\vartheta(t) +q(t) +\vartheta_c(t) \|_{L^p} &\le M\eta \|R_0(t)\|_{L^1}^{1/p}
		\label{eq:densityLp}\\
		\|w(t)+w_c(t)\|_{L^{p'}} &\le \frac{M}{\eta} \|R_0(t)\|_{L^1}^{1/p'}
		\label{eq:velocityLp}\\
		\|w(t)+w_c(t)\|_{W^{1,\tilde{p}}} &\le \delta
		\label{eq:velocityW1p}\\
		\left\| \left(
		 R^{\mathrm{time}, 1} +  R^{\mathrm{quadr}} + \div R^{\chi}  + R^{\mathrm{time}, 2}+  R^{\mathrm{lin}} + R^q + R^{\mathrm{corr}}
		\right)(t) \right\|_{L^1} &\le \delta.
		\label{eq:defectL1}
		\end{align}
	\end{subequations}
	
	\begin{rem}
		In all these definitions the oscillation parameter $ \lambda\in\N $ is still to be fixed. It will be chosen sufficiently large in the following estimates. Note that this is possible as there is no upper bound on $ \lambda $ here. 
	\end{rem}

	\subsection{Estimates on the perturbations}
Set
\begin{equation*}
A:= \big\{ t \in [0,T] \ : \ \|R_0(t)\|_{L^\infty} < \delta/4d \big\}, \qquad B: = [0,T] \setminus A. 
\end{equation*}
Since $R_0$ is a smooth function, $A$ is open in $[0,T]$ and thus $B$ is compact. It must then hold
\begin{equation*}
\inf_{t \in B} \|R_0(t)\|_{L^1} = \min_{t \in B} \|R_0(t)\|_{L^1} > 0.
\end{equation*}
If $t \in A$, then $\chi_j(t) \equiv 0$ for every $j$ and thus, by definition, $\vartheta(t) = q(t) = \vartheta_c(t) = w(t) = w_c(t) = 0$. Hence, \eqref{eq:densityLp} trivially holds. 
	If $t \in B$,  Lemmata~\ref{lem:thetaLp}, \ref{lem:qLp} and~\ref{lem:thetaC-qC} provide the desired bound on the density perturbation:
	\begin{align*}
		\|\vartheta(t)+q(t)+\vartheta_c(t)+q_c(t)\|_{L^p} &\le \|\vartheta(t)\|_{L^p} + \|q(t)\|_{L^p} + |\vartheta_c(t)| + |q_c(t)| \\
		&\le \frac{M\eta}{2} \|R_0(t)\|_{L^1}^{1/p} + C\left(\frac{1}{\lambda^{1/p}}+\frac{\mu^b}{\omega}+\frac{1}{\mu^b} + \frac{1}{\omega}\right) \\
		&= \frac{M\eta}{2} \|R_0(t)\|_{L^1}^{1/p} + C\left(\lambda^{-1/p} +\lambda^{b\alpha-\beta}+\lambda^{-b\alpha}+\lambda^{-\beta}\right).
	\end{align*}
	Because of \eqref{eq:beta-large} and the facts $ p<\infty $ and $ b>0 $ the second summand can be made arbitrarily small by choosing $ \lambda $ sufficiently large. More precisely, we can choose $\lambda$ so that
	\[ C\left(\lambda^{-1/p} +\lambda^{b\alpha-\beta}+\lambda^{-b\alpha}+\lambda^{-\beta}\right) < \frac{M\eta}{2} \min_{t \in B} \|R_0(t)\|_{L^1}^{1/p}, \]
which, in particular, proves \eqref{eq:densityLp}. Notice that, taking the minimum of the $\|R_0(t)\|_{L^1}$, we ensure that $\lambda$ can be chosen \emph{independent of $t$}. 
	
	For the $ L^{p'} $-bound on the velocity perturbation we need Lemmata~\ref{lem:wLp'} and~\ref{lem:wcLp'}.
\begin{align*}
	\|w+w_c\|_{L^{p'}} &\le \frac{M}{2\eta} \|R_0\|_{L^1}^{1/p'} + C \left( \frac{1}{\lambda^{1/p'}}+\sum_{k=1}^N \left(\frac{\mu\lambda}{\nu}\right)^k +\frac{(\mu\lambda)^{N+1}}{\nu^N} \right) \\
	&= \frac{M}{2\eta} \|R_0\|_{L^1}^{1/p'} + C \left( \lambda^{-1/p'}+\sum_{k=1}^N \left(\lambda^{1+\alpha-\gamma}\right)^k + \lambda^{(N+1)(1+\alpha)-N\gamma} \right).
\end{align*}
	Because of \eqref{eq:gamma-large} we have $ \lambda^{1+\alpha-\gamma}<1 $, so the sum inside the parentheses is bounded by $ N\lambda^{1+\alpha-\gamma} $. Furthermore
	\[ (N+1)(1+\alpha)-N\gamma < N(1+\alpha)-(N-1)\gamma < 0 \]
	holds by \eqref{eq:N-large}. Observe also that $ p'<\infty $, so all the exponents of $ \lambda $ in the parentheses are negative so the term can be made arbitrarily small by choosing $ \lambda $ sufficiently large, which proves \eqref{eq:velocityLp}.
	
	For \eqref{eq:velocityW1p} we apply Lemmata~\ref{lem:wW1p} and~\ref{lem:wcW1p} and obtain
	\begin{align*}
		\|w+w_c\|_{W^{1,\tilde{p}}} &\le C \left( \frac{\lambda\mu+\nu}{\mu^{1+\epsilon}} \right) \left( 1 + \sum_{k=1}^N \left(\frac{\lambda\mu}{\nu}\right)^k + \frac{(\lambda\mu)^{N+1}}{\nu^N} \right) \\
		&= C\left( \lambda^{1-\alpha\epsilon} + \lambda^{\gamma-\alpha(1-\epsilon)} \right) \left( \sum_{k=0}^N \lambda^{k(1+\alpha-\gamma)} +\lambda^{(N+1)(1+\alpha)-N\gamma} \right).
	\end{align*}
	Again because of \eqref{eq:gamma-large} and \eqref{eq:N-large} each summand inside the second parentheses is bounded by 1, so the inequality boils down to
	\[ \|w+w_c\|_{W^{1,\tilde{p}}} \le C(N+2)\left(\lambda^{1-\alpha\epsilon} + \lambda^{\gamma-\alpha(1-\epsilon)} \right). \]
	Both exponents of $\lambda$ in this expression are negative: The first one is by condition \eqref{eq:alpha-large} and the second by \eqref{eq:gamma-small}. Therefore, if $\lambda$ is large enough, \eqref{eq:velocityW1p} holds.
	
	\subsection{Estimates on the new error}
	By Lemma~\ref{lem:Rchi} the smoothness corrector term $ R^\chi $ is bounded in $ L^1 $ by $ \frac{\delta}{2} $ so in order to prove \eqref{eq:defectL1} we need to show that the sum of all other components of the defect field $ R_1 $ is smaller than $ \frac{\delta}{2} $ in $ L^1 $. Most of the terms are bounded analog to the density and velocity perturbations, by Lemmata~\ref{lem:Rquadr1}, \ref{lem:Rlin1} and~\ref{lem:Rqq1}:
	\begin{align*}
	\left\|R^{\mathrm{quadr}}\right\|_{L^1} &\le C \left(\lambda^{1+\alpha-\gamma}+\lambda^{-1}\right), \\
	\left\|R^{\mathrm{lin}}\right\|_{L^1} &\le C \left(\lambda^{-a\alpha} + \lambda^{-b\alpha}\right), \\
	\left\|R^q\right\|_{L^1} &\le C \lambda^{b\alpha-\beta}, \\
	\left\|R^{\mathrm{time},1}\right\|_{L^1} &\le C \lambda^{-\beta}.
	\end{align*}
	These terms are small for large $ \lambda $ because of \eqref{eq:gamma-large} (first line), as $ a,b>0 $ (second line) and by \eqref{eq:beta-large} (third and fourth line).
	
	The two remaining terms require more attention. Lemma~\ref{lem:Rcorr} provides the following bound on $ R^{\textrm{corr}} $:
	\[ \left\|R^{\textrm{corr}}\right\|_{L^1} \le C\left(1+\lambda^{-1/p}+\lambda^{b\alpha-\beta}\right) \left(\sum_{k=1}^N \lambda^{k(1+\alpha-\gamma)} + \lambda^{(N+1)(1+\alpha)-N\gamma}\right). \]
	By \eqref{eq:beta-large} the term in the first parentheses is bounded by 3, the second one is small for large $ \lambda $ because of \eqref{eq:gamma-large} and \eqref{eq:N-large} by the same argument as above in the estimate of the velocity perturbation. The last remaining term is $ R^{\mathrm{time},2} $, which is taken care of in Lemma~\ref{lem:Rlin2}:
	\begin{align*}
		\left\|R^{\mathrm{time},2}\right\|_{L^1} &\le C \lambda^{\beta-b\alpha} \left(\sum_{k=1}^N \lambda^{k(1+\alpha-\gamma)} + \lambda^{(N+1)(1+\alpha)-N\gamma}\right) \\
		&= C \lambda^{\beta+1+\alpha-(b\alpha+\gamma)} \left(\sum_{k=0}^{N-1} \lambda^{k(1+\alpha-\gamma)} + \lambda^{N(1+\alpha)-(N-1)\gamma}\right).
	\end{align*}
	Now \eqref{eq:gamma-large} and \eqref{eq:N-large} implies that the parentheses is bounded by $ N +1 $. Moreover the exponent $ \beta+1+\alpha-(b\alpha+\gamma) $ is negative because of condition \eqref{eq:beta-small}, so the term is arbitrarily small if $ \lambda $ is chosen sufficiently large. This concludes the proof of \eqref{eq:defectL1} and thus the proof of the proposition.

\section{Sketch of the proof of Proposition \ref{prop:main} for $p=1$ \\ and of Theorems~\ref{thm:diffusion} and~\ref{thm:higherReg}}
\label{sec:appendix}

\subsection{The case of continuous vector fields}
\label{subsec:continuous-vf}

The proof of Proposition~\ref{prop:main} at some points requires an integrability of the density perturbation $ \vartheta $ which is strictly better than $ L^1 $, most crucially in Lemma~\ref{lem:Rlin1}: Smallness for the term $ \|\vartheta u_0\|_{L^1} $ is impossible in the construction of the perturbation as presented in the previous sections.

In \cite{modena-szekelyhidi18} the same problem was solved by letting the Mikados ``deform with the flow'' so that the transport term in the linear part of $ R_1 $,
\[ \div R^{\mathrm{transport}} = \left(\partial_t+ u_0\cdot\nabla\right) \left(\vartheta-\fint_{\Td}\vartheta\right) \]
is sufficiently small because of a cancellation in the Mikado function.

More precisely, since $ u_0 $ is smooth, there exists the ``inverse flow map'', a smooth function $ \Phi:[0,T]\times\Td\to\Td $ which solves
\[ \partial_t \Phi + u_0 \cdot \nabla \Phi =0 \ ,\ \Phi(t_0,x) =x .\]
Moreover $ \Phi(t,\cdot):\Td\to\Td $ is close to the identity if $ t $ is close to $ t_0 $. In \cite{modena-szekelyhidi18} the perturbations are now defined using the pushforward of the Mikado density and flow. Ignoring corrector and cut-offs and using our notation the density perturbations locally in time has the representation
\[ \vartheta(t,x) = \eta \sum_j R_0^j(t,x) \Theta^j_{\lambda,\mu}\left( \Phi(t,x)\right). \]
It is easy to see that from this definition the transport term in the new defect field reduces to
\[ \left(\partial_t + u_0(t,x) \cdot\nabla\right) \vartheta(t,x) = \eta \sum_j \Theta_{\lambda,\mu}^j \left(\Phi(t,x)\right) \left(\partial_t + u_0(t,x)\cdot\nabla\right) R_0^j(t,x), \]
whose antidivergence is of order $1/\lambda$ in $L^1$-norm, because of the fast oscillating Mikado $\Theta^j_{\lambda, \mu}$.

In the construction presented in \Cref{sec:perturbation} it is advantageous to apply the pushforward only on the fast oscillating factor $ \psi(\nu x) $ and not on the space-time Mikado functions $ \varphi^j(t,x) $, which ensure the disjoint support where necessary. The density perturbation then takes the form
\[ \vartheta(t,x) = \eta\sum_j R_0^j(t,x) \varphi_\mu^j(\lambda(x-\omega te_j)) \psi^j \left(\nu \Phi(t,x)\right). \]
On the one hand the transport term also contains derivatives of $ (\varphi_\mu)_\lambda $, which excludes the possibility of a cheap $ L^1 $-estimate. However, the term is almost identical to $ \partial_t \vartheta $, so it is possible to estimate its antidivergence analog to Lemma~\ref{lem:Rlin2}. On the other hand, since the definition of the space-time Mikado functions $ \varphi^j(t,x) $ remains untouched, we still have disjoint support of Mikados in different directions, so there will not be any nontrivial interactions (``Third issue'' in Section~2 of \cite{modena-szekelyhidi18}) which need to be controlled.

All the other estimates in \Cref{sec:perturbation,sec:error,sec:proof} remain valid under this redefinition, so Proposition~\ref{prop:main} can be proved with $ p=1 $. For the technical details see \cite{modena-szekelyhidi18}.

\subsection{Handling the diffusion term}

In order to prove Theorem~\ref{thm:diffusion} we only need to add minor adjustments and one more estimate to the proof presented in \Cref{sec:error,sec:perturbation,sec:proof,sec:tools}.
The cheapest way to prove that $ (\rho_n,u_n,R_n) $ converges to a solution of \eqref{eq:transport-diffusion} is by showing that $ \nabla\rho_n $ converges in $ L^1 $. This way we can keep the construction of the perturbations untouched and just add $ \nabla(\rho_n-\rho_{n-1}) $ to the new defect field $ R_n $. Then clearly
\[ \partial_t \rho_1 +\div(\rho_1 u_1) - \Delta\rho_1 = -\div(R_1)-\Delta\rho_1 = -\div\left(R_1+\nabla\rho_1\right) \]
holds and it suffices to show that $ \nabla(\rho_1-\rho_0) $ is small in $ L^1 $. This estimate is straightforward: with the notation introduced in Section~\ref{sec:perturbation} we obtain
\begin{alignat*}{2}
	\|\nabla\vartheta\|_{L^1} &\le C \frac{1+\lambda\mu+\nu}{\mu^b} &&= C\frac{1+\lambda^{1+\alpha}+\lambda^\gamma}{\lambda^{b\alpha}},\\
	\|\nabla q\|_{L^1} &\le C \frac{1+\lambda\mu+\nu}{\omega} &&= C\frac{1+\lambda^{1+\alpha}+\lambda^\gamma}{\lambda^\beta}.
\end{alignat*}
(and trivially $ \nabla \vartheta_c=0 $.) We need to redefine $ \epsilon $ so that
\[ 0<\epsilon < \min\left\{ \frac{d}{\tilde{p}} - \frac{d}{p'} -1, \frac{d}{p'}-1 \right\}, \]
which is always possible by the additional condition \eqref{eq:p-cond-diffusion} in the statement of the theorem. Choose the parameters $ \alpha,\beta,\gamma $ exactly as before and observe that
\[ b>1+\epsilon \implies b\alpha>\alpha(1+\epsilon)>\gamma>1+\alpha \]
by conditions \eqref{eq:gamma-small} and \eqref{eq:alpha-large} and therefore $ \|\nabla\vartheta\|_{L^1} $ is small for large $ \lambda $. Similarly, $ \|\nabla q\|_{L^1} $ is also small as by \eqref{eq:beta-large} in particular $ \beta>\gamma,1+\alpha $.
This concludes the proof of an analog of Proposition \ref{prop:main} in the viscous case and thus Theorem~\ref{thm:diffusion}.

\subsection{Solutions of higher regularity}
Also for Theorem~\ref{thm:higherReg} the already existing proof requires only some adjustments and more estimates. For the sake of completeness and in order to motivate the extra conditions in the statement we state the analog of the main proposition.
\begin{prop}
	\label{prop:higherReg}
	There is constant $M>0$ such that the following holds. Let $ p,\tilde{p} \in [1,\infty) $ and $ m,\tilde{m}\in\N $ such that \eqref{eq:p-cond-higher} holds. There is $ s\in(1,\infty) $ such that for any $ \delta>0 $ and any smooth solution $ (\rho_0,u_0,R_0) $ of
	\begin{align*}
	\partial_t \rho +\div\left(\rho u\right) + L_k\rho &= -\div R, \\
	\div u &= 0,
	\end{align*}
	there is another smooth solution $ (\rho_1,u_1,R_1) $ which fulfils for any $ t\in[0,T] $
	\begin{subequations}
		\begin{align}
		\|\rho_1(t)-\rho_0(t)\|_{L^s} &\le M \|R_0(t)\|^{1/s} \\
		\|u_1(t)-u_0(t)\|_{L^{s'}} &\le M \|R_0(t)\|^{1/s'} \\
		\|\rho_1(t)-\rho_0(t)\|_{W^{m,p}} &\le\delta \\
		\|u_1(t)-u_0(t)\|_{W^{\tilde{m},\tilde{p}}} &\le\delta \\
		\|\rho_1(t)-\rho_0(t)\|_{W^{k-1,1}} &\le \delta
		\label{eq:Lk-rho-L1}\\
		\|R_1(t)\|_{L^1} &\le \delta \\
		R_0(t)\equiv 0 \implies R_1(t) &\equiv 0.
		\end{align}
	\end{subequations}
\end{prop}

\begin{proof}[Proof of Theorem~\ref{thm:higherReg}]
	For the differential operator of order $ k $ $ L_k $ there is an operator $ \tilde{L}_k $ such that
	\[ L_k f = \div\tilde{L}_k f \text{ for any smooth } f. \]
	Observe that $ \|\tilde{L}f\|_{L^r} \lesssim \|f\|_{W^{k-1,r}} $, so \eqref{eq:Lk-rho-L1} in particular implies
	\[ \left\|\tilde{L}_k (\rho_1-\rho_0) \right\|_{L^1} \le \delta. \]
	This guarantees that $R_n(t) \to 0$ in $L^1$, uniformly in time. 
	Completely analog to the proof of Theorem~\ref{thm:weak} we construct a sequence $ (\rho_n,u_n,R_n) $ of smooth solutions satisfying the bounds
	\begin{align*}
		\|\rho_{n+1}(t)-\rho_n(t)\|_{L^s} \le M \|R_n(t)\|^{1/s} &\le M\delta_{n-1}^{1/s} \\
		\|u_{n+1}(t)-u_n(t)\|_{L^{s'}} \le M \|R_n(t)\|^{1/s'} &\le M\delta_{n-1}^{1/s'} \\
		\|\rho_{n+1}(t)-\rho_n(t)\|_{W^{m,p}} &\le\delta_n \\
		\|u_{n+1}(t)-u_n(t)\|_{W^{\tilde{m},\tilde{p}}} &\le\delta_n \\
		\|\rho_{n+1}(t)-\rho_n(t)\|_{W^{k-1,1}} &\le \delta_n \\
		\|R_{n+1}(t)\|_{L^1} &\le \delta_n \\
		R_n(t)\equiv0 \implies R_{n+1}(t) &\equiv 0
	\end{align*}
	for $ (\rho_0,u_0)=(\bar{\rho},\bar{u}) $ and a sequence of positive numbers $ (\delta_n)_{n\in\N} $ chosen such that
	\[ \sum_{n\in\N} \delta_n^{1/s}<\infty  \ ,\ \sum_{n\in\N} \delta_n^{1/s'}<\infty \ ,\ \sum_{n\in\N} \delta_n<\infty, \]
and, in addition,
	\[ M \sum_{n\in\N} \delta_n^{1/s}<\e \]
	if we want to show (iv) or 
	\[ M \sum_{n\in\N} \delta_n^{1/s'}<\e \]	
	if we want to show (iv'). 	Then the limit 
	\[ \rho_n \xrightarrow{n\to\infty} \rho \text{ in }C\left([0,T],W^{m,p}(\Td)\right),\ u_n \xrightarrow{n\to\infty} u \text{ in }C\left([0,T],W^{\tilde{m},\tilde{p}}(\Td)\right) \]
	fulfils statements (i)--(iv) of the theorem.
\end{proof}

We only give a sketch of the proof of Proposition~\ref{prop:higherReg}, as it is mostly analog to the proof of Proposition~\ref{prop:main}. The only important difference is that in general $ u_1\in L^{p'} $ does not hold, which is needed for the $ L^1 $-convergence of the product $ \rho_n u_n $ and we want the density perturbation to be small in the Sobolev space $ W^{m,p} $, which was not necessary before. We address both issues by defining the Mikados in a slightly different way: The ``concentration scaling'' of Mikado density $ \Theta_\lambda $ and Mikado field $ W_\lambda $ is now given by
\[ \varphi_\mu(x) = \mu^a\varphi(\mu x),\ \tilde{\varphi}_\mu(x) = \mu^b \varphi(\mu x) \text{ where } a=\frac{d}{s},\ b=\frac{d}{s'} \]
for $ s\in(1,\infty) $ chosen such that
\[ \frac{1}{p}-\frac{m}{d} > \frac{1}{s} = 1-\frac{1}{s'} > 1+\frac{\tilde{m}}{d}-\frac{1}{\tilde{p}} \text{ and }
\frac{1}{s'} > \frac{k-1}{d}. \]
Note that such an $ s $ must exist because of \eqref{eq:p-cond-higher}.

With a suitable $ M $ and positive numbers $ \epsilon_1,\epsilon_2,\epsilon_3 $ defined as
\begin{align*}
\epsilon_1 & \coloneqq \frac{d}{m}\left(\frac{1}{p}-\frac{1}{s}\right)-1 \\
\epsilon_2 & \coloneqq \frac{d}{\tilde{m}} \min \left\{ \frac{1}{\tilde{p}}-\frac{1}{s'} , \frac{1}{\tilde{p}}-\frac{k-1}{d} \right\}-1 \\
\epsilon_3 & \coloneqq \frac{d}{s'(k-1)} -1
\end{align*}
the scaling of the Mikados implies
\begin{align*}
\left\|\Theta_{\lambda,\mu,\omega,\nu}\right\|_{L^s}, \left\|W_{\lambda,\mu,\omega,\nu}\right\|_{L^{s'}} &\le M \\
\left\|Q_{\lambda,\mu,\omega,\nu}\right\|_{L^s} &\lesssim \frac{\mu^b}{\omega} \\
\left\|\Theta_{\lambda,\mu,\omega,\nu}\right\|_{W^{m,p}} &\lesssim \left(\frac{\lambda\mu+\nu}{\mu^{(1+\epsilon_1)}}\right)^m \\
\left\|Q_{\lambda,\mu,\omega,\nu}\right\|_{W^{m,p}} &\lesssim \frac{(\lambda\mu+\nu)^m\mu^{d/p'}}{\omega} \\
\left\|W_{\lambda,\mu,\omega,\nu}\right\|_{W^{\tilde{m},\tilde{p}}} &\lesssim \left(\frac{\lambda\mu+\nu}{\mu^{(1+\epsilon_2)}}\right)^{\tilde{m}} \\
\left\|\Theta_{\lambda,\mu,\omega,\nu}\right\|_{W^{k-1,1}} &\lesssim \left(\frac{\lambda\mu+\nu}{\mu^{1+\epsilon_3}}\right)^{k-1} \\
\left\| Q_{\lambda,\mu,\omega,\nu}\right\|_{W^{k-1,1}} &\lesssim \frac{(\lambda\mu+\nu)^{k-1}}{\omega}.
\end{align*}
Choosing the parameters $ \mu=\lambda^\alpha $, $ \omega=\mu^\beta $ and $ \nu=\lambda^\gamma $ dependent of $ b $ and $ \epsilon \coloneqq \min \{\epsilon_1,\epsilon_2,\epsilon_3\}$ according to \eqref{eq:parameters} the proof of all necessary estimates is analog to those in Sections~\ref{sec:perturbation},~\ref{sec:error} and~\ref{sec:proof}.

\bibliographystyle{acm}
\bibliography{transport}

\begin{thebibliography}{10}

\bibitem{Ambrosio:2004cva}
{\sc Ambrosio, L.}
\newblock {Transport equation and Cauchy problem for BV vector fields}.
\newblock {\em Invent. math. 158}, 2 (2004), 227--260.

\bibitem{Ambrosio2017}
{\sc Ambrosio, L.}
\newblock {Well posedness of ODE’s and continuity equations with nonsmooth
  vector fields, and applications}.
\newblock {\em Rev. Mat. Complut. 30}, 3 (2017), 427--450.

\bibitem{Bianchini:2017vf}
{\sc Bianchini, S., and Bonicatto, P.}
\newblock {A uniqueness result for the decomposition of vector fields in Rd}.
\newblock {\em SISSA\/} (2017).

\bibitem{buckmaster-colombo-vicol18}
{\sc Buckmaster, T., Colombo, M., and Vicol, V.}
\newblock {Wild solutions of the Navier-Stokes equations whose singular sets in
  time have Hausdorff dimension strictly less than 1}.
\newblock {\em arXiv:1809.00600\/} (2018).

\bibitem{Buckmaster:2017uz}
{\sc Buckmaster, T., De~Lellis, C., Sz{\'e}kelyhidi~Jr, L., and Vicol, V.}
\newblock {Onsager's conjecture for admissible weak solutions}.
\newblock {\em arXiv\/} (2017).

\bibitem{Buckmaster:2017wf}
{\sc Buckmaster, T., and Vicol, V.}
\newblock {Nonuniqueness of weak solutions to the Navier-Stokes equation}.
\newblock {\em Annals of Mathematics\/} (2019).

\bibitem{Caravenna:2016kg}
{\sc Caravenna, L., and Crippa, G.}
\newblock {Uniqueness and Lagrangianity for solutions with lack of
  integrability of the continuity equation}.
\newblock {\em C. R. Math. Acad. Sci. Paris 354}, 12 (2016), 1168--1173.

\bibitem{caravenna-crippa18}
{\sc Caravenna, L., and Crippa, G.}
\newblock {A directional Lipschitz extension lemma, with applications to
  uniqueness and Lagrangianity for the continuity equation}.
\newblock {\em arXiv:1812.06817\/} (2018).

\bibitem{cheskidov19}
{\sc Cheskidov, A., and Luo, X.}
\newblock {Stationary and discontinuous weak solutions of the Navier-Stokes
  equations}.
\newblock {\em arXiv:1901.07485\/} (2019).

\bibitem{SzekelyhidiJr:2016tp}
{\sc Daneri, S., and Sz{\'e}kelyhidi~Jr, L.}
\newblock {Non-uniqueness and h-principle for H\"older-continuous weak
  solutions of the Euler equations}.
\newblock {\em Arch. Rational Mech. Anal. 224}, 2 (2017), 471--514.

\bibitem{Depauw:2003wl}
{\sc Depauw, N.}
\newblock {Non unicit{\'e} des solutions born{\'e}es pour un champ de vecteurs
  BV en dehors d'un hyperplan}.
\newblock {\em C. R. Math. Acad. Sci. Paris 337}, 4 (2003), 249--252.

\bibitem{DiPerna:1989vo}
{\sc DiPerna, R.~J., and Lions, P.-L.}
\newblock {Ordinary differential equations, transport theory and Sobolev
  spaces}.
\newblock {\em Invent. math. 98}, 3 (1989), 511--547.

\bibitem{gilbart01}
{\sc Gilbarg, D., and Trudinger, N.~S.}
\newblock {\em {Elliptic Partial Differential Equations of Second Order}}.
\newblock Springer-Verlag Berlin Heidelberg, 2001.

\bibitem{Isett:2016to}
{\sc Isett, P.}
\newblock {A Proof of Onsager's Conjecture}.
\newblock {\em arXiv\/} (2016).

\bibitem{titi18}
{\sc Luo, T., and Titi, E.~S.}
\newblock {Non-uniqueness of Weak Solutions to Hyperviscous Navier-Stokes
  Equations -- On Sharpness of J.-L. Lions Exponent}.
\newblock {\em arXiv:1808.07595\/} (2018).

\bibitem{luo18}
{\sc Luo, X.}
\newblock {Stationary solutions and nonuniqueness of weak solutions for the
  Navier-Stokes equations in high dimensions}.
\newblock {\em arXiv:1807.09318\/} (2018).

\bibitem{modena-szekelyhidi18}
{\sc Modena, S., and Sz{\'e}kelyhidi~Jr, L.}
\newblock Non-renormalized solutions to the continuity equation.
\newblock {\em arXiv:1806.09145\/} (2018).

\bibitem{Modena2018}
{\sc Modena, S., and Sz{\'e}kelyhidi~Jr, L.}
\newblock {Non-uniqueness for the transport equation with Sobolev vector
  fields}.
\newblock {\em Annals of PDE 4}, 2 (2018), 18.

\bibitem{SEIS20171837}
{\sc Seis, C.}
\newblock A quantitative theory for the continuity equation.
\newblock {\em Annales de l'Institut Henri Poincaré C, Analyse non linéaire
  34}, 7 (2017), 1837 -- 1850.

\end{thebibliography}

\end{document}